\documentclass[reqno,11.5pt]{amsart}

\usepackage{diagrams}
\usepackage{amssymb}
\usepackage{tikz}
\usepackage{amsmath}
\usepackage{latexsym}
\usepackage{mathrsfs}
\usepackage{amsthm}
\usepackage{verbatim}
\usepackage{graphicx}
\usepackage{epstopdf}
\usepackage{epsfig}
\diagramstyle[labelstyle=\scriptstyle]
\usepackage{color}
\usepackage{float}
\usepackage{bm}
\usepackage[colorlinks,linkcolor=blue,anchorcolor=green,citecolor=red]{hyperref}
\usepackage{amsfonts}
\usepackage{fancyhdr}
\usepackage{amscd}

\usepackage[utf8]{inputenc}

\usepackage{url}
\usepackage{hyperref}
\definecolor{red}{rgb}{1,0,0}

\definecolor{blue}{rgb}{0,0,1}

\newtheorem{theorem}{Theorem}
\newtheorem{remark}{Remark}

\newtheorem{lemma}{Lemma}
\newtheorem{problem}{Problem}

\newtheorem{algorithm}{Algorithm}

\newcommand\grad{\operatorname{grad}}
\renewcommand\div{\operatorname{div}}
\newcommand\curl{\operatorname{curl}}

\renewcommand\S{{S}}

\makeatletter
\@addtoreset{equation}{section}
\makeatother

\usepackage{geometry}
\begin{document}

\title{Convergence of a $\bm{B}$-$\bm{E}$ based finite element method for MHD models on Lipschitz domains}

\thanks{The title of this article in its original version (see arXiv:1711.11330) is ``Magnetic-Electric 
Formulations for Stationary Magnetohydrodynamics Models".}

\thanks{The work of Kaibo Hu was partly carried out during his affiliation with the University of Oslo, supported by the  European Research Council under the European Union's 
Seventh Framework Programme (FP7/2007-2013) / ERC grant agreement 339643. Weifeng Qiu is partially 
supported by a grant from the Research Grants Council of the Hong Kong Special Administrative Region, China 
(Project No. CityU 11304017).  As a convention the names of the authors are alphabetically ordered. 
all authors contributed equally in this article. }

\author{Kaibo Hu} 
\address{School of Mathematics,
University of Minnesota, Vincent Hall, 206 Church St. SE, Minneapolis, MN, USA.}
\email{khu@umn.edu}
\author{Weifeng Qiu}
\address{Department of Mathematics, City University of Hong Kong, 83 Tat Chee Avenue, Hong Kong, China.}
\email{weifeqiu@cityu.edu.hk}
\author{Ke Shi}
\address{Department of Mathematics and Statistics, Old Dominion University, Norfolk, VA 23529, USA.}
\email{kshi@odu.edu}


\maketitle

\begin{abstract}
We discuss a class of magnetic-electric fields based finite element schemes for stationary magnetohydrodynamics (MHD) systems with two types of boundary conditions.  We establish a key $L^{3}$ estimate for divergence-free finite element functions for a new type of boundary conditions. With this estimate and a similar one in \cite{hu2015structure}, we rigorously prove the convergence of Picard iterations and the finite element schemes with weak regularity assumptions. These results demonstrate the convergence of the finite element methods for singular solutions. 
\end{abstract}

\section{Introduction}   
Magnetohydrodynamics (MHD) models have various important applications in liquid metal industry, controlled fusion and astronomy etc. There have been extensive discussions on numerical methods for MHD models. However, due to the nonlinear coupling and rich structures of MHD systems, the numerical simulation still remains a challenging and active research area.  This paper is devoted to the analysis of a class of stable and structure-preserving finite element methods.

We consider the following stationary MHD system on a polyhedral domain $\Omega$:
\begin{align}
\begin{cases}
& ( \bm{u} \cdot \nabla) \bm{u}
- R_{e}^{-1} \Delta \bm{u}
- \S \bm{j} \times \bm{B}
+ \nabla p
= \bm{f} , \\
& \bm{j}
- R_{m}^{-1} \nabla \times \bm{B}
= \bm{0} , \\
& \nabla \times \bm{E} = \bm{0} , \\
& \nabla \cdot \bm{B} = 0, \\
& \nabla \cdot \bm{u} = 0, \\
& \bm{j} = \bm{E} + \bm{u} \times \bm{B}.
\end{cases}
\label{eq:MHD-stationary}
\end{align}
Here $\bm{u}$ is the fluid velocity, $p$ is the fluid pressure, $\bm{j}$ is the current density, $\bm{E}$ and $\bm{B}$ are the electric and magnetic fields respectively. 

We mainly consider the following type of boundary conditions:
\begin{equation}
\label{boundary_cond}
\bm{u} = 0, \quad \bm{B} \cdot \bm{n} = 0, \quad \bm{E} \times \bm{n} = \bm{0},\quad \text{on $\partial \Omega$},
\end{equation}
where $\bm{n}$ is the unit normal vector of $\partial \Omega$.
We also consider an alternative boundary condition c.f., \cite{gunzburger1991existence}: 
\begin{align}
\label{boundary_cond2}
\bm{u} = 0, \quad \bm{B} \times \bm{n} = \bm{0}, \quad \bm{E} \cdot \bm{n} = 0,\quad \text{on $\partial \Omega$}.
\end{align}

Finite element discretizations of the MHD system \eqref{eq:MHD-stationary} have a long history. Based on the function and finite element spaces for the magnetic variable $\bm{B}$, these schemes can be classified as $H^{1}$-, $H(\curl)$- and $H(\div)$-based formulations. 
Gunzburger \cite{gunzburger1991existence} studied a finite element method where $\bm{B}$ was discretized in $H^{1}$ with the Lagrange elements. With certain conditions on the boundary data and right hand side, Gunzburger \cite{gunzburger1991existence} proved the existence and uniqueness of the weak solutions and established optimal error estimates for the finite element methods. The domain is assumed to be bounded in $\mathbb{R}^{3}$ which is either convex or has a $C^{1, 1}$ boundary. Under this assumption, the true solution is smooth. 
And the convergence proof in \cite{gunzburger1991existence} also relies on this assumption. 
To remove this restriction on the domain, Sch\"{o}tzau \cite{Schotzau.D.2004a} proposed another variational formulation with the magnetic variable in $H(\curl)$. In the finite element scheme based on this formulation,  $\bm{B}$ is discretized in the $H(\curl)$-conforming N\'{e}d\'{e}lec spaces \cite{Nedelec.J.1980a,Nedelec.J.1986a} and the quasi-optimal convergence of the approximation solutions was shown in \cite{Schotzau.D.2004a}.   We refer to, e.g., \cite{dong2014convergence,greif2010mixed,gerbeau2000stabilized,badia2013unconditionally,shadid2016scalable}, for some variants and the convergence analysis of iterative methods and finite element discretizations based on the these two approaches. 

For MHD systems, magnetic Gauss's law plays an important role in both physics (nonexistence of magnetic monopole) and numerical simulations (c.f., \cite{Brackbill.J;Barnes.D.1980a,Dai.W;Woodward.P.1998b}). However, in the above $H(\curl)$ based approach, magnetic Gauss's law is only preserved in the weak sense.  One way to obtain schemes with precisely preserved magnetic Gauss's law is to use the vector potential of $\bm{B}$, see \cite{adler2018vector,hiptmair2018fully,hiptmair2018splitting} and the references therein. Since the vector potential belongs to $H(\curl)$, this method also falls in the category of $H(\curl)$ based formulations. 

To preserve magnetic Gauss's law precisely on the discrete level with electric and magnetic fields as variables, a class of finite element schemes was developed in \cite{hu2014stable,hu2015structure} for the time dependent and the stationary MHD systems respectively. The magnetic field $\bm{B}$ is discretized by the $H(\div)$ conforming Raviart-Tomas \cite{Raviart.P;Thomas.J.1977a} or BDM \cite{brezzi1985two} elements.   An electric variable, either the electric field $\bm{E}$ in \cite{hu2014stable} or the current density $\bm{j}$ in \cite{hu2015structure}, is retained in the formulation and discretized by the $H(\curl)$ conforming elements in the same discrete de Rham complex.

In this paper, we prove the convergence of the $H(\div)$ based methods for stationary MHD problems with weak regularity assumptions. Several variants of this type of schemes exist,  and we choose to consider a $\bm{B}$-$\bm{E}$ based formulation in the discussions below. This formulation is the stationary case of \cite{hu2014stable} and differs from the $\bm{B}$-$\bm{j}$ formulation in \cite{hu2015structure} by a projection of nonlinear terms (see Section \ref{sec:vatiational-formulation} below for details). Therefore we do not claim the discretization studied in this paper as a brand new method, although the precise formulation has not appeared in the literature to the best of our knowledge. 

To show the convergence with both types of boundary conditions, we 
extend the key Hodge mapping and $L^{3}$ estimates established in \cite{hu2015structure} to 
a new type of boundary condition. 
With an analysis based on the reduced systems, we show that the schemes are unconditional 
stable and well-posed.
We prove the convergence of the finite element scheme by carefully choosing interpolation functions (see \eqref{div0_curl0} below). Comparing with the convergence analysis in \cite{hu2015structure} for the $\bm{B}$-$\bm{j}$ based finite element methods, we adopt a new strategy and, as a result, only assume weak regularity of the solutions in this paper (\eqref{regularity} below).
This demonstrates the convergence of the Picard iterations and the finite element schemes for singular solutions.


We also show another strategy to impose the strong divergence-free condition, instead of using Lagrange multipliers as in the previous work \cite{hu2015structure} by one of the authors and collaborator. We introduce an augmented term $(\nabla\cdot\bm{B}, \nabla\cdot\bm{C})$ in the variational formulation. Thanks to the structure-preserving properties, these two approaches are actually equivalent and Faraday's law $\nabla\cdot\bm{B}=0$ also holds precisely on the discrete level.

The remaining part of this paper will be organized as follows. In Section $2$, we provide some preliminary settings. 
In Section $3$, we give two types of $L^{3}$ estimates for the discrete magnetic field. In Sections $4,5$ and $6$, 
we formulate the numerical method for the MHD models with boundary condition (\ref{boundary_cond}), 
show that its Picard iterations are well-posed and convergent, and show the optimal convergence of approximations 
to the velocity field and magnetic field even for singular solutions. In Section $7$, we generalize 
the numerical method for the MHD models with boundary condition (\ref{boundary_cond2}), provide its 
basic properties and show the optimal convergence.

\section{Preliminaries}

{ We assume that $\Omega$ is a bounded Lipschitz
  polyhedron. For the ease of exposition, we further assume that
  $\Omega$ is contractable, i.e. there is no nontrivial harmonic
  form. 

 Using the standard notation for the inner product and the norm of the
{$L^{2}$} space
$$
(u,v):=\int_{\Omega}u\cdot v \mathrm{d}x,\quad
\|u\|:=\left(\int_{\Omega} \lvert u\rvert^2 \mathrm{d}x\right)^{1/2}. 
$$
The scalar function space $H^{1}$ is defined by
$$
H^1(\Omega):=\left\{v\in L^{2}(\Omega): \nabla\bm{v}\in L^{2}(\Omega)\right \}. 
$$
For a function $u \in W^{k,p}(\Omega)$, we use $\|u\|_{k,p}$ for the standard norm in $W^{k,p}(\Omega)$. When $p = 2$ we drop the index $p$, i.e. $\|u\|_k := \|u\|_{k,2}$ and $\|u\|:= \|u\|_{0,2}$. We define vector function spaces 
$$
H(\curl,\Omega):=\{\bm{v}\in L^2(\Omega), \nabla\times \bm{v}\in L^2(\Omega)\},  
$$
and
$$
H(\div,\Omega):=\{\bm{w}\in L^2(\Omega), \nabla\cdot \bm{w}\in L^2(\Omega)\}.
$$
With explicit boundary conditions, we define
$$
H_{0}^1(\Omega):=\left\{v\in H^{1}(\Omega): v\left |_{\partial \Omega} =0\right .\right \},
$$
$$
H_0(\curl,\Omega):=\{\bm{v}\in H(\curl, \Omega), \bm{v}\times \bm{n}=0 \mbox{ on } \partial\Omega\},
$$
and
$$
H_0(\div,\Omega):=\{\bm{w}\in H(\div, \Omega), \bm{w}\cdot \bm{n}=0 \mbox{ on } \partial\Omega\}.
$$
 We often use the following notation:
$$
L^2_0(\Omega):=\left\{v\in L^2(\Omega):  \int_\Omega v=0 \right\}.
$$
The corresponding norms in $H^{1}$, $H(\curl)$ and $H(\div)$ spaces are defined by
$$
\|\bm{u}\|_{1}^{2}=\|\bm{u}\|^{2}+\|\nabla \bm{u}\|^{2},
$$
$$
\|\bm{F}\|_{\mathrm{curl}}^{2}:=\|\bm{F}\|^{2}+\|\nabla\times \bm{F}\|^{2},
$$
$$
\|\bm{C}\|_{\mathrm{div}}^{2}:= \|\bm{C}\|^{2}+\|\nabla\cdot \bm{C}\|^{2}.
$$

  For a general Banach space
$\bm{Y}$ with a norm $\|\cdot\|_{\bm{Y}}$, the dual space
$\bm{Y}^{\ast}$ is equipped with the dual norm defined by
$$
\|\bm{h}\|_{\bm{Y}^{\ast}}:=\sup_{0 \neq \bm{y}\in \bm{Y}}\frac{\langle \bm{h}, \bm{y} \rangle}{\|\bm{y}\|_{\bm{Y}}}.
$$
For the special case that $\bm{Y}=H_0^1(\Omega)$, the dual space
$\bm{Y}^\ast=H^{-1}(\Omega)$ and the corresponding norm is denoted by
$\|\cdot\|_{-1}$, which is defined by
$$
\|\bm{f}\|_{-1}:=\sup_{0 \neq \bm{v}\in [{H}_{0}^{1}(\Omega)]^{3}}\frac{\langle \bm{f}, \bm{v} \rangle}{\|\nabla\bm{v}\|}.
$$

In this paper, we will use $C$ to denote a generic constant in inequalities which is independent of the exact solution and the mesh size. For instance, we will need the following Poincar\'{e}'s inequality:
\begin{align}\label{def:C1}
\|{u}\|_{0,6}\leq C\|\nabla{u}\|,\quad \forall ~ {u}\in H^{1}_{0}(\Omega).
\end{align}

Since the fluid convection frequently appears in subsequent
discussions, we introduce a trilinear form
$$
L(\bm{w}; \bm{u}, \bm{v}):=\frac{1}{2}[((\bm{w}\cdot \nabla)\bm{u},
\bm{v})- ((\bm{w}\cdot\nabla) \bm{v}, \bm{u}) ].
$$
Considering $\bm{w}$ as a known function, $L(\bm{w}; \bm{u}, \bm{v})$ is a
bilinear form of $\bm{u}$ and $\bm{v}$. 

Let $\mathcal{T}_{h}$ be a triangulation of $\Omega$, and we assume
that the mesh is regular and quasi-uniform, so that the inverse
estimates hold \cite{brenner2008mathematical}.  We use $P_{k}(\mathcal{T}_{h})$ to denote the piecewise polynomial space of degree $k$ on $\mathcal{T}_{h}$. The finite element de
Rham sequence is an abstract framework to unify the above spaces and
their discretizations, see e.g. Arnold, Falk, Winther
\cite{Arnold.D;Falk.R;Winther.R.2006a,Arnold.D;Falk.R;Winther.R.2010a},
Hiptmair \cite{Hiptmair.R.2002a}, Bossavit \cite{Bossavit.A.1998a} for
more detailed discussions. Figure~\ref{exact-sequence}  and Figure~\ref{exact-sequence-nobc} 
show the commuting diagrams we will use. The electric field $\bm{E}$ and the magnetic
field $\bm{B}$  will be discretized in the middle two spaces respectively. 
Notice that though projections in Figure~\ref{exact-sequence} can be different from 
corresponding ones in Figure~\ref{exact-sequence-nobc}, we don't need to distinguish them 
in any analysis in this paper.
\begin{figure}[ht!]
\begin{equation*}
\begin{CD}
H_0(\mathrm{grad})   @> {\mathrm{grad}} >> H_0(\mathrm{curl}) 
@>{\mathrm{curl}} >> H_0(\mathrm{div})  @> {\mathrm{div}} >> L_0^2  \\    
 @VV\Pi^{\mathrm{grad}} V @VV\Pi^{\mathrm{curl}} V @VV\Pi^{\mathrm{div}} V
@VV\Pi^0 V\\  
H^h_0(\mathrm{grad})  @>{\mathrm{grad}} >> H^h_0(\mathrm{curl})  @>{\mathrm{curl}} >> H^h_0(\mathrm{div}) @> {\mathrm{div}} >> L^{2,h}_0
\end{CD}
\end{equation*}
\caption{Continuous and discrete de Rham sequence - homogeneous boundary conditions}
\label{exact-sequence}
\end{figure}
\begin{figure}[ht!]
\begin{equation*}
\begin{CD}
H(\mathrm{grad})   @> {\mathrm{grad}} >> H(\mathrm{curl}) 
@>{\mathrm{curl}} >> H(\mathrm{div})  @> {\mathrm{div}} >> L^2  \\    
 @VV\Pi^{\mathrm{grad}} V @VV\Pi^{\mathrm{curl}} V @VV\Pi^{\mathrm{div}} V
@VV\Pi^0 V\\  
H^h(\mathrm{grad})  @>{\mathrm{grad}} >> H^h(\mathrm{curl})  @>{\mathrm{curl}} 
>> H^h(\mathrm{div}) @> {\mathrm{div}} >> L^{2,h}
\end{CD}
\end{equation*}
\caption{Continuous and discrete de Rham sequence - no boundary condition}
\label{exact-sequence-nobc}
\end{figure}

As we shall see, $H(\mathrm{div})$ functions with vanishing divergence will play an important role in the study. So we define on the continuous level
$$
H_{0}(\mathrm{div}0, \Omega):=\{\bm{C}\in H_{0}(\mathrm{div}, \Omega): \nabla\cdot\bm{C}=0  \},
$$
and the finite element subspace
$$
H_{0}^{h}(\mathrm{div}0, \Omega):=\{\bm{C}_{h}\in H_{0}^{h}(\mathrm{div}, \Omega): \nabla\cdot\bm{C}_{h}=0  \}.
$$

 We use $\bm{V}_{h}$ to denote the finite element
subspace of velocity $\bm{u}_{h}$, and $Q_{h}$ for pressure
$p_{h}$. There are many existing stable pairs for $\bm{V}_{h}$ and
$Q_{h}$, for example, the Taylor-Hood elements
\cite{Girault.V;Raviart.P.1986a, Boffi.D;Brezzi.F;Fortin.M.2013a}.
Spaces $H^{h}_{0}(\mathrm{div}, \Omega)$ and $ L_{0}^{2, h}(\Omega)$ are finite
elements from the discrete de Rham
sequence. For these spaces we use the explicit names for clarity, and use the notation 
$\bm{V}_{h}$ and $\bm{Q}_{h}$ for the fluid part to indicate that they may be different from 
$H_{0}^{h}(\grad, \Omega)$ and $L^{2, h}_{0}(\Omega)$ in the de Rham sequence. We use $\bm{V}_{h}^{0}$ to denote the discrete velocity space, i.e.
$$
\bm{V}_{h}^{0}:=\left \{\bm{v}_{h}\in \bm{V}_{h}: (\nabla\cdot\bm{v}_{h}, q_{h})=0, ~\forall q_{h}\in Q_{h}\right \}.
$$

There is a unified theory for the discrete de Rham sequence  of arbitrary
order \cite{Boffi.D;Brezzi.F;Fortin.M.2013a,
  Arnold.D;Falk.R;Winther.R.2006a, Arnold.D;Falk.R;Winther.R.2010a}.
In the case $n = 3$, the lowest order elements can be represented as:
\begin{equation*}
\begin{diagram}
\mathbb{R} & \rTo^{\subset} &  \mathcal{P}_{3}\Lambda^{0} & \rTo^{d} & \mathcal{P}_{2} \Lambda^{1}& \rTo^{d} &\mathcal{P}_{1} \Lambda^{2}& \rTo^{d} & \mathcal{P}_{0} \Lambda^{3}& \rTo & 0,\\
\mathbb{R} & \rTo^{\subset} & \mathcal{P}_{2}\Lambda^{0} & \rTo^{d} &\mathcal{P}_{1} \Lambda^{1}& \rTo^{d} &\mathcal{P}_{1}^{-} \Lambda^{2}& \rTo^{d} & \mathcal{P}_{0} \Lambda^{3} & \rTo & 0,\\
\mathbb{R} & \rTo^{\subset} &  \mathcal{P}_{2}\Lambda^{0}  & \rTo^{d} &\mathcal{P}_{2}^{-} \Lambda^{1}& \rTo^{d} &\mathcal{P}_{1} \Lambda^{2}& \rTo^{d} & \mathcal{P}_{0} \Lambda^{3}& \rTo & 0,\\
\mathbb{R} & \rTo^{\subset} &   \mathcal{P}_{1}\Lambda^{0}  & \rTo^{d} &\mathcal{P}_{1}^{-} \Lambda^{1}& \rTo^{d} &\mathcal{P}_{1}^{-} \Lambda^{2} & \rTo^{d} &  \mathcal{P}_{0} \Lambda^{3}& \rTo&  0.
\end{diagram}
\end{equation*}
The correspondence between the language of differential forms and classical finite element methods is summarized in Table \ref{tab:mixed-fem}. 

To obtain compatible finite element schemes, below we require that the discrete spaces 
 $H_{0}^{h}(\mathrm{curl}, \Omega)$, $H_{0}^{h}(\mathrm{div}, \Omega)$ and $L^{2, h}_{0}(\Omega)$ belong to the same finite element de Rham sequence. 
  \begin{table}[H]
\centering
\begin{tabular}{c  c l}
\hline\noalign{\smallskip}
$k$ & $\Lambda_{h}^{k}(\Omega)$ & Classical finite element space \\[0.5ex]
\noalign{\smallskip}\hline\noalign{\smallskip}
0 & $\mathcal{P}_{r} \Lambda^{0}(\mathcal{T})$ & Lagrange elements of degree $\leq r$ \\
1 & $\mathcal{P}_{r} \Lambda^{1}(\mathcal{T})$ & Nedelec 2nd-kind $H(\mathrm{curl})$ elements of degree $\leq r$ \\
2 & $\mathcal{P}_{r} \Lambda^{2}(\mathcal{T})$ & Nedelec 2nd-kind $H(\mathrm{div})$ elements of degree $\leq r$ \\
3 & $\mathcal{P}_{r} \Lambda^{3}(\mathcal{T})$ & discontinuous elements of degree $\leq r$ \\
\noalign{\smallskip}\hline\noalign{\smallskip}
0 & $\mathcal{P}_{r}^{-} \Lambda^{0}(\mathcal{T})$ & Lagrange elements of degree $\leq r$ \\
1 & $\mathcal{P}_{r}^{-} \Lambda^{1}(\mathcal{T})$ & Nedelec 1st-kind $H(\mathrm{curl})$ elements of order $r-1$ \\
2 & $\mathcal{P}_{r}^{-} \Lambda^{2}(\mathcal{T})$ & Nedelec 1st-kind $H(\mathrm{div})$ elements of order $r-1$ \\
3 & $\mathcal{P}_{r}^{-} \Lambda^{3}(\mathcal{T})$ & discontinuous elements of degree $\leq r-1$ \\[1ex]
\noalign{\smallskip}\hline
\end{tabular}
\caption{Correspondences between finite element differential forms and the classical finite element spaces for $n = 3$ (from \cite{Arnold.D;Falk.R;Winther.R.2006a})}\label{tab:mixed-fem}
\end{table}

As we shall see, it is useful to group the spaces to define 
$$
\bm{X}_{h}:=\bm{V}_{h}\times 
H^{h}_{0}(\mathrm{curl}, \Omega)\times H_{0}^{h}(\mathrm{div},
\Omega).
$$
and group $Q_{h}\times L_{0}^{2, h}(\Omega)$ to define
$$
\bm{Y}_{h}:=Q_{h}\times L_{0}^{2, h}(\Omega).
$$

For the analysis, we also need to define a reduced space, where $\bm{E}_{h}$ is eliminated:
$$
\bm{W}_{h}:=\bm{V}_{h}\times H^{h}_{0}(\mathrm{div}, \Omega).
$$

Denote the kernel space
$$
\bm{X}^{00}_{h}:=\bm{V}_{h}^{0}\times 
H^{h}_{0}(\mathrm{curl}, \Omega)\times H_{0}^{h}(\mathrm{div}0, \Omega),
$$
and
$$
\bm{W}_{h}^{00}:=\bm{V}_{h}^{0}\times H^{h}_{0}(\mathrm{div}0, \Omega).
$$
By definition, any $(\bm{u}_{h}, \bm{B}_{h})\in \bm{W}^{00}_{h}$ satisfies $(\nabla\cdot\bm{u}_{h}, q_{h})=0, ~~\forall q_{h}\in Q_{h}$ and  $\nabla\cdot\bm{B}_{h}=0$.

In order to define appropriate norms, we introduce the discrete curl operator on the discrete
level. For any $\bm{C}_{h}\in H^{h}_{0}(\mathrm{div}, \Omega)$, define
$\nabla_{h}\times \bm{C}_{h}\in H^{h}_{0}(\mathrm{curl}, \Omega)$ by:
\begin{equation}\label{dis_curl}
(\nabla_{h}\times \bm{C}_{h}, \bm{F}_{h})=(\bm{C}_{h}, \nabla\times \bm{F}_{h}), \quad \forall \bm{F}_{h}\in H^{h}_{0}(\mathrm{curl},\Omega).
\end{equation}
For any $\bm{w}_{h}\in H_{0}^{h}(\mathrm{curl}, \Omega)$, we define $\nabla_{h}\cdot\bm{w}_{h}\in H_{0}^{h}(\mathrm{grad}, \Omega)$ by
\begin{equation}\label{dis_div}
(\nabla_{h}\cdot \bm{w}_{h}, v_{h})=-(\bm{w}_{h}, \nabla v_{h}), \quad \forall v_{h}\in H_{0}^{h}(\mathrm{grad}, \Omega).
\end{equation}

We define $\mathbb{P}: L^{2}(\Omega)\rightarrow H^{h}_{0}(\mathrm{curl},\Omega)$ to be the $L^{2}$ projection
$$
(\mathbb{P}\phi, \bm{F}_{h})=(\phi, \bm{F}_{h}), \quad\forall \bm{F}_{h}\in H^{h}_{0}(\mathrm{curl},\Omega), \phi\in L^{2}(\Omega).
$$
We further define $\|\cdot\|_{d}$ as a modified norm of $H^{h}_{0}(\mathrm{div}, \Omega)$ by
$$
\|\bm{C}_{h}\|_{d}^{2}:=\|\bm{C}_{h}\|^{2}+\|\nabla\cdot \bm{C}_{h}\|^{2}+\|\nabla_{h}\times \bm{C}_{h}\|^{2}.
$$ 

Now we define the norms for various product spaces.   For space $\bm{Y}_{h}$,
we define
\begin{align}\label{norm-Y}
\|(q,r)\|_{\bm{Y}}^{2}:=\|q\|^{2}+\|r\|^{2}.
\end{align}
For other product spaces,   we define 
\begin{align}\label{stationary-BE-norm-X}
\|(\bm{v}, \bm{F}, \bm{C})\|_{\bm{X}}^{2}:=\|\bm{v}\|^{2}+\|\nabla \bm{v}
\|^{2}+\|\nabla\times \bm{F} \|^{2}&+\|\bm{F}+\bm{v}\times \bm{B}^{-}\|^{2}+\|\bm{C}\|^{2}+\|\nabla\cdot\bm{C}\|^{2},\\&\nonumber
\forall(\bm{v}, \bm{F}, \bm{C})\in \bm{X}_{h},
\end{align}
and
$$
\|(\bm{u}_{h}, \bm{B}_{h})\|_{\bm{W}}^{2}:=\|\bm{u}_{h}\|_{1}^{2}+\|\bm{B}_{h}\|_{d}^{2}, \quad \forall (\bm{u}_{h}, \bm{B}_{h})\in {\bm{W}}_{h}.
$$
Here $\bm{B}^{-}\in H(\div, \Omega)$ is a given function.

The constant $sR_{m}^{-1}$ will appear in the discussions below frequently, therefore we denote
$$
\alpha:=sR_{m}^{-1}.
$$

\section{Hodge mapping and $L^{p}$ estimates for divergence-free finite elements} 
 
In this section we present some key $L^3$ embedding results which are crucial for our analysis in the following sections.

\begin{theorem}\label{L3estimate}
For any function $\bm{d}_h \in H_0^h(\div0, \Omega)$, we have
\[
\|\bm{d}_h\|_{0,3} \le C \|\nabla_h \times \bm{d}_h\|,
\]
where the generic constant $C$ solely depends on $\Omega$.
\end{theorem} 
Theorem \ref{L3estimate} and its proof can be found in \cite[Theorem 1]{hu2015structure}.
For the boundary condition given in \eqref{boundary_cond2}, we have similar estimates.
\begin{theorem}\label{L3estimate_new}
For any function $\bm{d}_h \in H^h(\div0, \Omega)$, we have
\[
\|\bm{d}_h\|_{0,3} \le C \|\tilde{\nabla}_h \times \bm{d}_h\|,
\]
where $\tilde{\nabla}_h \times \bm{d}_h \in H^{h}(\curl, \Omega)$ satisfies
\begin{align*}
(\tilde{\nabla}_h \times \bm{d}_h, \bm{F}) = (\bm{d}_{h}, \nabla\times \bm{F}), 
\quad \forall \bm{F} \in H^{h}(\curl, \Omega).
\end{align*}
The generic constant $C$ solely depends on $\Omega$.
\end{theorem} 
 
\begin{proof}
We define $Z_{0} = H_{0}(\curl, \Omega)\cap H(\div0, \Omega)$, 
$Z_{0}^{h} = H^{h}(\div, \Omega) \cap H(\div0, \Omega)$. 
Obviously, $\bm{d}_{h} \in Z_{0}^{h}$. 
We define an operator $H_{d}: Z_{0}^{h} \rightarrow Z_{0}$ by 
\begin{align*}
(\nabla\times (H_{d}\bm{d}_{h}), \nabla\times \bm{v}) = (\tilde{\nabla}_{h}\times \bm{d}_{h}, \nabla\times \bm{v}), 
\quad \forall \bm{v} \in Z_{0}.
\end{align*}
Obviously, $H_{d}$ is well defined. Since $H_{d}\bm{d}_{h} \in Z_{0}$, we have 
\begin{align}
\label{H_d_prop1}
\Vert H_{d}\bm{d}_{h}\Vert_{\frac{1}{2}+\delta} \leq C \Vert \nabla\times (H_{d}\bm{d}_{h})\Vert 
\leq C \Vert \tilde{\nabla}_{h}\times \bm{d}_{h} \Vert,
\end{align}
where $\delta \in (0,\frac{1}{2}]$. 

We use the projections $\Pi^{\curl}$ and $ \Pi^{\div} $ 
in the commuting diagram in Figure \ref{exact-sequence-nobc}.

Since $\nabla \cdot (\bm{d}_{h} -  \Pi^{\div} (H_{d}\bm{d}_{h})) = 0$ in $\Omega$, there exists 
$\bm{\phi}_{h} \in \{ \bm{v} \in H^{h}(\curl, \Omega): (\bm{v}, \nabla s)=0, 
\quad \forall s \in H^{h}(\grad,\Omega)\}$, such that 
\begin{align*}
\nabla\times \bm{\phi}_{h} = \bm{d}_{h} -  \Pi^{\div} (H_{d}\bm{d}_{h}).
\end{align*}

We consider the auxiliary problem: 
\begin{align}
\label{aux_problem1}
\nabla\times \nabla\times \bm{\psi} = & \nabla\times \bm{\phi}_{h} \quad \text{ in } \Omega, \\
\nonumber 
\nabla\cdot \bm{\psi} = & 0 \quad \text{ in } \Omega, \\
\nonumber 
\bm{\psi}\times \bm{n} = & \bm{0} \quad \text{ on } \partial\Omega.
\end{align}
Since $\nabla\cdot (\nabla\times \bm{\phi}_{h}) = 0$ in $\Omega$, the auxiliary problem (\ref{aux_problem1}) 
is well-posed. Obviously, $\nabla\times \bm{\psi}$ satisfies
\begin{align*}
\nabla\times (\nabla\times \bm{\psi}) = & \nabla\times \bm{\phi}_{h} \quad \text{ in } \Omega,\\ 
\nabla\cdot (\nabla\times \bm{\psi}) = & 0 \quad \text{ in } \Omega,\\
(\nabla\times \bm{\psi})\cdot \bm{n} = & 0 \quad \text{ on } \partial\Omega.
\end{align*}

According to \cite[Lemma~$4.2$]{Hiptmair.R.2002a}, we have 
\begin{align}
\label{aux_prop1}
\Vert \nabla\times \bm{\psi} \Vert_{\frac{1}{2}+\delta} \leq C \Vert \nabla\times \bm{\phi}_{h}\Vert 
= C \Vert \bm{d}_{h} -  \Pi^{\div} (H_{d}\bm{d}_{h})\Vert.
\end{align}

We claim that 
\begin{align}
\label{aux_prop2}
\Vert \nabla\times \bm{\psi} - \bm{\phi}_{h}\Vert \leq C h^{\frac{1}{2}+\delta} 
\Vert \bm{d}_{h} -  \Pi^{\div} (H_{d}\bm{d}_{h})\Vert.
\end{align}
Notice that by (\ref{aux_problem1}), 
\begin{align*}
\nabla\times  \Pi^{\curl} (\nabla\times \bm{\psi}) =  \Pi^{\div}  (\nabla\times \nabla\times \bm{\psi}) 
=  \Pi^{\div}  (\nabla \times \bm{\phi}_{h}) = \nabla \times \bm{\phi}_{h}.
\end{align*}
Since $ \Pi^{\curl} (\nabla\times \bm{\psi}), \bm{\phi}_{h} \in H^{h}(\curl, \Omega)$, 
there exists $s_{h} \in H^{h}(\grad, \Omega)$ such that 
\begin{align*}
 \Pi^{\curl} (\nabla\times \bm{\psi}) - \bm{\phi}_{h} = \nabla s_{h}\quad \text{ in } \Omega. 
\end{align*}
Since $(\nabla\times \bm{\psi})\cdot \bm{n} = 0$ on $\partial\Omega$, we have 
\begin{align*}
(\nabla\times \bm{\psi},  \Pi^{\curl} (\nabla\times \bm{\psi}) - \bm{\phi}_{h}) 
= (\nabla\times \bm{\psi}, \nabla s_{h}) = 0.
\end{align*}
By the construction of $\bm{\phi}_{h}$, we have 
\begin{align*}
(\bm{\phi}_{h},  \Pi^{\curl} (\nabla\times \bm{\psi}) - \bm{\phi}_{h}) 
= (\bm{\phi}_{h}, \nabla s_{h}) = 0.
\end{align*}
Thus 
\begin{align*}
(\nabla \times \bm{\psi} - \bm{\phi}_{h},  \Pi^{\curl} (\nabla\times \bm{\psi}) - \bm{\phi}_{h}) = 0.
\end{align*}
So, by the above identify and (\ref{aux_prop1}), we have 
\begin{align*}
 \Vert \nabla \times \bm{\psi} - \bm{\phi}_{h}\Vert 
\leq & \Vert (\nabla \times \bm{\psi} - \bm{\phi}_{h})
 - ( \Pi^{\curl} (\nabla\times \bm{\psi}) - \bm{\phi}_{h}) \Vert \\ 
= & \Vert \nabla \times \bm{\psi} -   \Pi^{\curl} (\nabla\times \bm{\psi})\Vert \\ 
\leq & C h^{\frac{1}{2}+\delta} \Vert \bm{d}_{h} -  \Pi^{\div} (H_{d}\bm{d}_{h})\Vert.
\end{align*}
Therefore, the claim (\ref{aux_prop2}) is correct.

By the construction of $H_{d}$ and the fact that $\bm{\psi} \in Z_{0}$, 
\begin{align*}
(\tilde{\nabla}_{h}\times \bm{d}_{h}, \nabla\times \bm{\psi}) 
= (\nabla\times (H_{d}\bm{d}_{h}), \nabla\times \bm{\psi}) 
= (H_{d}\bm{d}_{h}, \nabla\times \nabla\times \bm{\psi}) 
= (H_{d}\bm{d}_{h}, \nabla\times \bm{\phi}_{h}).
\end{align*}
By the fact that $\bm{\phi}_{h}\in H^{h}(\curl,\Omega)$ and the above identity, 
\begin{align*}
& (\bm{d}_{h}, \nabla\times \bm{\phi}_{h}) = (\tilde{\nabla}_{h}\times \bm{d}_{h}, \bm{\phi}_{h}) \\
= & (\tilde{\nabla}_{h}\times \bm{d}_{h}, \bm{\phi}_{h} - \nabla\times \bm{\psi}) 
+ (\tilde{\nabla}_{h}\times \bm{d}_{h}, \nabla\times \bm{\psi}) \\
= & (\tilde{\nabla}_{h}\times \bm{d}_{h}, \bm{\phi}_{h} - \nabla\times \bm{\psi}) 
+ (H_{d}\bm{d}_{h}, \nabla\times \bm{\phi}_{h}).
\end{align*}
Thus we have 
\begin{align*}
(\bm{d}_{h} - H_{d}\bm{d}_{h}, \bm{d}_{h} -  \Pi^{\div}  (H_{d}\bm{d}_{h})) 
= (\bm{d}_{h} - H_{d}\bm{d}_{h}, \nabla\times \bm{\phi}_{h}) 
= (\tilde{\nabla}_{h}\times \bm{d}_{h}, \bm{\phi}_{h} - \nabla\times \bm{\psi}).
\end{align*}
So we have 
\begin{align*}
& \Vert \bm{d}_{h} - H_{d} \bm{d}_{h} \Vert^{2} \\
= & (\bm{d}_{h} - H_{d} \bm{d}_{h}, \bm{d}_{h} -  \Pi^{\div} ( H_{d} \bm{d}_{h})) 
+ (\bm{d}_{h} - H_{d} \bm{d}_{h},  \Pi^{\div} ( H_{d} \bm{d}_{h}) - H_{d} \bm{d}_{h}) \\
= & (\tilde{\nabla}_{h}\times \bm{d}_{h}, \bm{\phi}_{h} - \nabla\times \bm{\psi}) 
+ (\bm{d}_{h} - H_{d} \bm{d}_{h},  \Pi^{\div} ( H_{d} \bm{d}_{h}) - H_{d} \bm{d}_{h}) \\ 
\leq & \Vert \tilde{\nabla}_{h}\times \bm{d}_{h}\Vert \cdot \Vert \bm{\phi}_{h} - \nabla\times \bm{\psi}\Vert 
+ \Vert \bm{d}_{h} - H_{d} \bm{d}_{h}\Vert \cdot \Vert  \Pi^{\div} ( H_{d} \bm{d}_{h}) - H_{d} \bm{d}_{h}\Vert. 
\end{align*}
By applying (\ref{aux_prop2}) in the above inequality, we have 
\begin{align}
\label{H_d_prop2}
\Vert \bm{d}_{h} - H_{d} \bm{d}_{h} \Vert \leq C h^{\frac{1}{2}+\delta} 
\Vert \tilde{\nabla}_{h}\times \bm{d}_{h}\Vert.
\end{align}

Let $k_{0}$ be a positive integer such that $H_0^h(\div0, \Omega) \subset [ P_{k_{0}}(\mathcal{T}_{h}) ]^{3}$. 
We denote by $\bm{\Pi}$ the standard $\bm{L}^{2}$-orthogonal projection onto $[ P_{k_{0}}(\mathcal{T}_{h}) ]^{3}$. 
Thus $\bm{\Pi} \bm{d}_{h} = \bm{d}_{h}$. So, by the discrete inverse inequality and the fact that 
$\Vert \bm{\Pi} \bm{v} \Vert_{0,3} \leq C \Vert \bm{v}\Vert_{0,3}$ for any $\bm{v} \in [ L^{3}(\Omega) ]^{3}$, we have 
\begin{align*}
\Vert \bm{d}_{h} \Vert_{0,3} = & \Vert \bm{\Pi} \bm{d}_{h} \Vert_{0,3}
\leq \Vert \bm{\Pi} (\bm{d}_{h} - H_{d} \bm{d}_{h}) \Vert_{0,3} + \Vert \bm{\Pi} ( H_{d} \bm{d}_{h}) \Vert_{0,3} \\ 
\leq & C \big( \Vert h^{-\frac{1}{2}} \bm{\Pi} (\bm{d}_{h} - H_{d} \bm{d}_{h}) \Vert 
+ \Vert H_{d} \bm{d}_{h} \Vert_{0,3} \big) \\ 
\leq & C \big( h^{\delta} \Vert \nabla_{h}\times \bm{d}_{h}\Vert 
+ \Vert H_{d}\bm{d}_{h}\Vert_{\frac{1}{2}+\delta} \big). 
\end{align*}
Since $H_{d}\bm{d}_{h} \in H_{0}(\curl,\Omega)\cap H(\div0,\Omega)$, 
\begin{align*}
\Vert H_{d}\bm{d}_{h}\Vert_{\frac{1}{2}+\delta} \leq C \Vert \nabla\times (H_{d} \bm{d}_{h}) \Vert 
\leq C \Vert \nabla_{h}\times \bm{d}_{h}\Vert. 
\end{align*}
So, we can conclude that 
\begin{align*}
\|\bm{d}_h\|_{0,3} \le C \|\nabla_h \times \bm{d}_h\|. 
\end{align*}

This completes the proof.

 \end{proof}
\section{Variational formulations}\label{sec:vatiational-formulation}

\subsection{Nonlinear scheme}
\label{sec_scheme}
We propose the following variational form for  \eqref{eq:MHD-stationary} with boundary condition (\ref{boundary_cond}):
\begin{problem}\label{prob:continuous}
Find $(\bm{u}_h, \bm{E}_h, \bm{B}_h)\in \bm{X}_{h}$ and $ (p_h, r_h)\in \bm{Y}_{h}$, such that for any $(\bm{v}, \bm{F}, \bm{C})\in \bm{X}_{h} $ and $ (q, s)\in \bm{Y}_{h}$,
\begin{subequations}\label{fem-discretization}
\begin{equation}\label{fem-1}
 L(\bm{u}_h; \bm{u}_h, \bm{v})
  + R_{e}^{-1} (\nabla \bm{u}_h, \nabla \bm{v})  
  - \S (\bm{j}_h\times \bm{B}_h,\bm{v} ) - (p_h,\nabla\cdot \bm{v}) 
 = \langle \bm{f},\bm{v} \rangle,\\
 \end{equation}
 \begin{equation}\label{fem-2}
\S(\bm{j}_h,\bm{F}) - \alpha ( \bm{B}_h, \nabla\times \bm{F}) =0, 
\end{equation}
 \begin{equation}\label{fem-3}
 \alpha (\nabla\times \bm{E}_h, \bm{C}) +(r_h, \nabla\cdot \bm{C})= 0, 
 \end{equation}
 \begin{equation}\label{fem-4}
 -(\nabla\cdot \bm{u}_h, q) =0,
 \end{equation}
 \begin{equation}\label{fem-5}
 (\nabla\cdot \bm{B}_h, s) = 0, 
 \end{equation}
\end{subequations}
where $\bm{j}_h$ is given by  Ohm's law: 
$\bm{j}_h = \bm{E}_h + \bm{u}_h \times \bm{B}_h$. Here $r_h$ is the Lagrange multiplier which approximates $r = 0$.
\end{problem} 

We verify some properties of the variational form Problem \ref{prob:continuous}:
\begin{theorem}\label{stability_fem}
Any solution for  Problem \ref{prob:continuous} satisfies
\begin{enumerate}
\item
magnetic Gauss's law:
$$
\nabla\cdot \bm{B}_h=0.
$$
\item  Lagrange multiplier  $r_h=0$, and the strong form
$$
\nabla\times \bm{E}_h =0,
$$
\item energy estimates:
\begin{align}\label{energy-1}
{R_{e}^{-1}}\|\nabla\bm{u}_h\|^{2}+\S\|\bm{j}_h\|^{2} &= \langle \bm{f}, \bm{u}_h\rangle, \\
\label{energy-2}
\frac{1}{2}R_{e}^{-1}\|\nabla\bm{u}_h\|^{2}+\S\|\bm{j}_h\|^{2} &\leq \frac{R_{e}}{2}\|\bm{f}\|_{-1}^{2}, \\
\label{energy-3}
R_m^{-1} \|\nabla_h \times \bm{B}_h\| &\le  \|\bm{j}_h\|, \\
\label{energy-4}
\Vert \nabla_h \times \bm{B}_h \Vert & \le C R_{e}^{\frac{1}{2}} R_{m} \S^{-\frac{1}{2}} \Vert \bm{f} \Vert_{-1}, \\
\label{energy-5}
\|\bm{E}_h\| &\le C R_e^{\frac32} R_m \S^{-\frac{1}{2}} \|\bm{f}\|^2_{-1}.
\end{align}
\end{enumerate}
\end{theorem}

\begin{proof} The magnetic Gauss's law is a direct consequence of \eqref{fem-5}.

Taking $\bm{C}=\nabla\times \bm{E}_h$ in  \eqref{fem-3}, we have
$\nabla\times \bm{E}_h=0$. Therefore \eqref{fem-3} reduces to
$$
(r_h, \nabla\cdot \bm{C})=0, \quad\forall\,  \bm{C}\in {H_{0}^{h}(\mathrm{div}, \Omega)}.
$$
Since ${L_{0,h}^2(\Omega)}= \nabla\cdot {H_{0}^{h}(\mathrm{div}, \Omega)}$, we get $r_h=0$. 

To obtain the first energy estimate, we take $\bm{v}=\bm{u}_h$, $\bm{F}=\bm{E}_h, \bm{C}=\bm{B}_h$ and $q = p_h$ in \eqref{fem-1} - \eqref{fem-4} and add the equations together.  The second energy estimate follows from the Young's inequality
$$
\langle \bm{f}, \bm{u}_{h} \rangle \leq \|\bm{f}\|_{-1}\|\nabla\bm{u}_{h}\|
\leq
\frac{1}{2R_{e}}\|\nabla\bm{u}_{h}\|^{2}+\frac{1}{2}R_{e}\|\bm{f}\|_{-1}^{2}.
$$
 
Taking $\bm{F} = \nabla_h \times \bm{B}_h$ in \eqref{fem-2} we have
\[
R^{-1}_m \|\nabla_h \times \bm{B}_h\|^2 = R^{-1}_m (\bm{j}_h, \nabla_h \times \bm{B}_h) \le \|\bm{j}_h\| \|\nabla_h \times \bm{B}_h\|,
\]
which implies \eqref{energy-3}.  Obviously, the estimate (\ref{energy-4}) is due to estimates (\ref{energy-2}) 
and (\ref{energy-3}).

Next we take $\bm{F} = \bm{E}_h$ in \eqref{fem-2} and by the definition of $\bm{j}_h$ we have
\[
(\bm{E}_h + \bm{u}_h \times \bm{B}_h, \bm{E}_h) - R_m^{-1} (\bm{B}_h, \nabla \times \bm{E}_h) = 0.
\]
By the fact that $\nabla \times \bm{E}_h = 0$ and the generalized H\"{o}lder's inequality we have
\begin{align}\nonumber
\|\bm{E}_h\|^2 &=  - (\bm{u}_h \times \bm{B}_h, \bm{E}_h) \le \|\bm{u}_h\|_{0,6} \|\bm{B}_h\|_{0,3} \|\bm{E}_h\| \\
\nonumber
& \le C\|\nabla \bm{u}_h\| \|\nabla_h \times \bm{B}_h\| \|\bm{E}_h\|,
\end{align}
the last step is due to the Sobolev embedding results \eqref{def:C1} and Theorem \ref{L3estimate}. The estimate 
\eqref{energy-5} can be obtained by combining the above estimate with \eqref{energy-2} and \eqref{energy-4}. 
This completes the proof.
\end{proof}

\begin{remark}\label{stab_const}
From the above result we can see that the energy norm of the unknowns $\bm{u}_h, \bm{B}_h, \bm{E}_h$ solely depends on $\|\bm{f}\|_{-1}$ and the physical constants $R_m, R_e, \S$. In addition, it is easy to verify that the exact solution satisfies the same stability estimate
\begin{align}
\label{stab_exact_solution}
\Vert \nabla \bm {u}\Vert & \leq R_{e} \Vert \bm{f} \Vert_{-1},\\
\nonumber 
\Vert \bm{B}\Vert_{0,3} + \Vert \nabla\times \bm{B} \Vert 
& \leq C R_{e}^{\frac{1}{2}}R_{m}\S^{-\frac{1}{2}}\Vert \bm{f} \Vert_{-1}, \\
\nonumber 
\Vert \bm{E} \Vert & \leq C R_e^{\frac32} R_m \S^{-\frac{1}{2}} \|\bm{f}\|^2_{-1}.
\end{align}
\end{remark}


\begin{theorem}
Problem \ref{prob:continuous} is well-posed.
\end{theorem}
In the remaining part of this section we prove the well-posedness of Problem \ref{prob:continuous}. We will first recast Problem \ref{prob:continuous} into an equivalent form (\eqref{BE-stationary} and Problem \ref{prob:mixed-picard}) where $\bm{E}$ is formally eliminated. Then we demonstrate that this equivalent form is well-posed using the Brezzi theory and the key $L^{3}$ estimate (Theorem \ref{thm:wellposed-stationary}). Then we can conclude with the well-posedness of Problem \ref{prob:continuous}.

Using \eqref{fem-2},  we have 
$$
\bm{E}_{h}+\mathbb{P}(\bm{u}_{h}\times \bm{B}_{h})=R_{m}^{-1}\nabla_{h}\times \bm{B}.
$$
Now  the Lorentz force has an equivalent form
\begin{align}\nonumber
-\left (\bm{j}_{h}\times \bm{B}_{h},\bm{v} \right )&=\left (\bm{E}_{h}+\mathbb{P}\left (\bm{u}_{h}\times \bm{B}_{h}\right ),  \bm{v}\times \bm{B}_{h}\right )+\left ((I-\mathbb{P})(\bm{u}_{h}\times \bm{B}_{h}),  \bm{v}\times \bm{B}_{h}\right )\\&\label{penalty-curl}
=R_{m}^{-1}\left (\nabla_{h}\times \bm{B}_{h},  \bm{v}\times \bm{B}_{h}\right )+\left((I-\mathbb{P})(\bm{u}_{h}\times \bm{B}_{h}), (I-\mathbb{P}) (\bm{v}\times \bm{B}_{h})\right).
\end{align}
Even though the velocity field $\bm{u}_{h}$ is smooth, the $H(\div)$ conformality of the magnetic field $\bm{B}_{h}$ cannot guarantee  $\bm{u}_{h}\times \bm{B}_{h}\in H(\curl, \Omega)$. 
 The term $(I-\mathbb{P})(\bm{u}_{h}\times \bm{B}_{h})$ on the right hand side of  \eqref{penalty-curl}  measures the deviation of $\bm{u}_{h}\times \bm{B}_{h}$ from   $H^{h}(\curl)$ and $\left((I-\mathbb{P})(\bm{u}_{h}\times \bm{B}_{h}), (I-\mathbb{P}) (\bm{u}_{h}\times \bm{B}_{h})\right)$ can be regarded as a penalty term.

Therefore \eqref{fem-discretization} is equivalent to the following problem: Find $\left ( \bm{u}_{h}, \bm{B}_{h}\right )\in \bm{W}_{h}$ and $(p_{h}, r_{h})\in \bm{Y}_{h}$ such that for any $\left ( \bm{v}_{h}, \bm{C}_{h}\right )\in \bm{W}_{h}$ and $(q_{h}, s_{h})\in \bm{Y}_{h}$,
\begin{equation}\label{BE-stationary}
\left \{
\begin{aligned}
&  L(\bm{w}_{h}; \bm{u}_{h}, \bm{v}_{h})+ R_{e}^{-1} (\nabla \bm{u}_{h}, \nabla \bm{v}_{h})
- {\alpha}(\nabla_{h}\times \bm{B}_{h} ,\bm{B}_{h}\times\bm{v}_{h} )\\&
\quad\quad+ \S\left((I-\mathbb{P})(\bm{u}_{h}\times \bm{B}_{h}), (I-\mathbb{P}) (\bm{v}_{h}\times \bm{B}_{h})\right)
- (p_{h},\nabla\cdot \bm{v}_{h}) = (\bm{f},\bm{v}_{h}),\\&
-{\alpha}(\bm{u}_{h}\times \bm{B}_{h}, \nabla_{h}\times \bm{C}_{h})+\S R_{m}^{-2}(\nabla_{h}\times \bm{B}_{h}, \nabla_{h}\times \bm{C}_{h})+(r_{h}, \nabla\cdot\bm{C}_{h})=0,\\&
(\nabla\cdot\bm{u}_{h}, q_{h})=0,\\&
 (\nabla\cdot \bm{B}_{h}, s_{h}) = 0.
\end{aligned}  
\right .
\end{equation}

We note that the reduced system \eqref{BE-stationary} has a similar form compared with the work by Gunzburger \cite{gunzburger1991existence} and Sch\"otzau \cite{Schotzau.D.2004a}. However, this similarity is only formal. The magnetic field $\bm{B}$ is discretized as 0-forms with the Lagrange finite elements in \cite{gunzburger1991existence} and treated as 1-forms with the N\'{e}d\'{e}lec elements in \cite{Schotzau.D.2004a}.  In both approaches \cite{gunzburger1991existence, Schotzau.D.2004a}, 
the $\curl$ operator can be evaluated on $\bm{B}$ in a straightforward way. In contrast, $\bm{B}$ is discretized as a 2-form in \eqref{BE-stationary}. As a result, the discrete curl operator $\nabla_{h}\times $ is globally defined by \eqref{dis_curl}, which leads to a new mixed formulation. This also makes the analysis essentially different from \cite{gunzburger1991existence} or  \cite{Schotzau.D.2004a}. Compared with the $\bm{B}$-$\bm{j}$ based scheme in \cite{hu2015structure}, a quadratic term 
$$
\S\left((I-\mathbb{P})(\bm{u}_{h}\times \bm{B}_{h}), (I-\mathbb{P}) (\bm{v}_{h}\times \bm{B}_{h})\right),
$$
comes into the reduced variational formulation \eqref{BE-stationary}.  This is due to the different choice of variables. 

Denote $\bm{\psi}_{h}=(\bm{w}_{h}, \bm{G}_{h})$, $\bm{\xi}_{h}=(\bm{u}_{h},  \bm{B}_{h})$, $\bm{\eta}_{h}=(\bm{v}_{h}, \bm{C}_{h})$ and $\bm{x}_{h}=(p_{h},r_{h})$, $\bm{y}_{h}=(q_{h}, s_{h})$.
Define
\begin{align*}
\bm{a}_{s}\left (\bm{\psi}_{h}; \bm{\xi}_{h},  \bm{\eta}_{h}\right )
 := & {1 \over 2}\left [\left (( \bm{w}_{h}\cdot \nabla) \bm{u}_{h},\bm{v}_{h}\right )
- \left ( (\bm{w}_{h}\cdot \nabla)\bm{v}_{h} , \bm{u}_{h}\right )\right ]+ R_{e}^{-1}\left  (\nabla \bm{u}_{h}, \nabla \bm{v}_{h}\right )\\&
- {\alpha}\left (\nabla_{h}\times \bm{B}_{h} ,\bm{G}_{h}\times\bm{v}_{h} \right )+ \S\left((I-\mathbb{P})(\bm{u}_{h}\times \bm{G}_{h}), (I-\mathbb{P}) (\bm{v}_{h}\times \bm{G}_{h})\right)\\&
-{\alpha}\left (\bm{u}_{h}\times \bm{G}_{h}, \nabla_{h}\times \bm{C}_{h}\right )+\S R_{m}^{-2}\left (\nabla_{h}\times \bm{B}_{h}, \nabla_{h}\times \bm{C}_{h}\right ),
\end{align*}
and
$$
\bm{b}_{s}(\bm{\xi}_{h}, \bm{y}_{h}):= -(\nabla\cdot \bm{u}_{h}, q_{h})+(\nabla\cdot \bm{B}_{h}, s_{h}).
$$

Equation \eqref{BE-stationary}  can be recast into a mixed system:
\begin{problem}\label{prob:mixed-picard}  
Given $\bm{\theta}\in \bm{W}_{h}^{\ast} $ and $\bm{\psi}\in \bm{Y}_{h}^{\ast}$, find $(\bm{\xi}_{h},\bm{x}_{h})\in
  \bm{W}_{h} \times \bm{Y}_{h}$, such that
\begin{align}\label{brezzip01}
\begin{cases}
&\bm{a}_{s}(\bm{\xi}_{h}; \bm{\xi}_{h},  \bm{\eta}_{h})+ \bm{b}_{s}(\bm{\eta}_{h},  \bm{x}_{h})=\langle \bm{\theta}, \bm{\eta}_{h}  \rangle ,
\quad \forall~       
 \bm{\eta}_{h} \in \bm{W}_{h}, \\
&\bm{b}_{s}(\bm{\xi}_{h}, \bm{y}_{h})=\langle \bm{\psi}, \bm{y}_{h} \rangle, \quad\forall \bm{y}_{h} \in \bm{Y}_{h}.
\end{cases}
\end{align}
\end{problem}

\begin{theorem} \label{thm:wellposed-stationary}
Problem \ref{prob:mixed-picard} is well-posed.
\end{theorem}

We prove the existence of solutions to the discrete variational form. To use the Brezzi theory and the fixed point theorem (see \cite{Girault.V;Raviart.P.1986a}), we need to show
\begin{itemize}
\item
each term in \eqref{brezzip01} is bounded,
\item
the inf-sup condition for $\bm{b}_{s}$,
\item
coercivity of $\bm{a}_{s}$ on $\bm{W}_{h}^{00}$.
\end{itemize}
We establish these conditions in the subsequent lemmas.

The boundedness of the variational form is a direct consequence of the key $L^{3}$ estimate.
\begin{lemma}\label{lem:boudnedness-stationary-BE}
The trilinear form $\bm{a}_{s}(\cdot;\cdot, \cdot)$ and the bilinear form $\bm{b}_{s}(\cdot, \cdot)$ are bounded, i.e. there exists a positive constant    $C$ such that
$$
\bm{a}_{s}(\bm{\psi}_{h}; \bm{\xi}_{h}, \bm{\eta}_{h})\leq C\|\bm{\psi}_{h}\|_{\bm{W}}\|\bm{\xi}_{h}\|_{\bm{W}}\|\bm{\eta}_{h}\|_{\bm{W}}, \quad \forall \bm{\psi}_{h}, \bm{\xi}_{h}, \bm{\eta}_{h}\in \bm{W}_{h}, 
$$
and
$$
\bm{b}_{s}(\bm{\eta}_{h}, \bm{y}_{h})\leq C\|\bm{\eta}_{h}\|_{\bm{W}}\|\bm{y}_{h}\|_{\bm{Y}}, \quad \forall \bm{\eta}_{h}\in \bm{W}_{h}, \bm{y}_{h}\in \bm{Y}_{h}.
$$
\end{lemma}

Since we have used a stronger norm for  $\bm{B}_{h}, \bm{C}_{h}\in H_{0}^{h}(\div, \Omega)$, 
the inf-sup condition for the bilinear form $\bm{b}_{s}(\cdot, \cdot)$ becomes more subtle. Following a similar proof as shown in \cite{hu2015structure} for the $\bm{B}$-$\bm{j}$ formulation, we get:
\begin{lemma}{\bf (inf-sup conditions for ${\bm{b}_{s}}(\cdot,\cdot)$)} \label{lem:inf-supb-BE-stationary}
There exists a positive constant $\gamma$ such that
$$
\inf_{\bm{y}_{h}\in \bm{Y}_{h}}\sup_{{\bm{\eta}}_{h}\in \bm{W}_{h}} \frac{{\bm{b}_{s}}( {\bm{\eta}}_{h}, \bm{y}_{h})}{\|{\bm{\eta}}_{h}\|_{\bm{W}}\|\bm{y}_{h}\|_{\bm{Y}}}\geq \gamma>0.
$$
\end{lemma}

The coercivity of $\bm{a}_{s}(\cdot; \cdot, \cdot)$ holds on the kernel space $\bm{W}_{h}^{00}$.
\begin{lemma}\label{lem:a-BE-stationary}
On $\bm{W}_{h}^{00}$ we have
$$
\bm{a}_{s}(\bm{\xi}_{h}; \bm{\xi}_{h}, \bm{\xi}_{h})\geq \gamma \|\bm{\xi}_{h}\|_{\bm{W}}^{2},
$$
where $\gamma$ is a positive constant.
\end{lemma}
\begin{proof}
We note that
$$
\bm{a}_{s}(\bm{\xi}_{h}; \bm{\xi}_{h}, \bm{\xi}_{h})=R_{e}^{-1}\|\nabla\bm{u}_{h}\|^{2}+\S \|(I-\mathbb{P})\left ( \bm{u}_{h}\times \bm{B}_{h}\right )\|^{2}+\S R_{m}^{-2}\|\nabla_{h}\times \bm{B}_{h}\|^{2}.
$$
Discrete  Poincar\'{e}'s inequality holds on $\bm{W}_{h}^{00}$:
$$
\|\bm{B}_{h}\|\lesssim \|\nabla_{h}\times \bm{B}_{h}\|.
$$
This completes the proof.
\end{proof}

By Lemma \ref{lem:boudnedness-stationary-BE}, Lemma \ref{lem:inf-supb-BE-stationary} and Lemma \ref{lem:a-BE-stationary},  the nonlinear variational form \eqref{brezzip01} is well-posed.  Therefore  \eqref{BE-stationary} has at least one solution. For suitable source and boundary data, the solution is also unique.

\subsection{Picard iterations}

We propose the following Picard type iterations for Problem  \ref{prob:continuous}:
\begin{algorithm}[Picard iterations for nonlinear schemes]
\label{alg:picard-s}
Given $(\bm{u}^{n-1},\bm{B}^{n-1})$,   find $(\bm{u}^{n}, \bm{E}^{n}, \bm{B}^{n})\in \bm{X}_{h}$ and $ (p^{n}, r^{n})\in \bm{Y}_{h}$, such that for any $(\bm{v}, \bm{F}, \bm{C})\in \bm{X}_{h} $ and $ (q, s)\in \bm{Y}_{h}$,
\begin{align}
\label{picard1}
 L(\bm{u}^{n-1}; \bm{u}^{n}, \bm{v}) 
  +{R_{e}^{-1}} (\nabla \bm{u}^{n}, \nabla \bm{v}) 
  - \S (\bm{j}^{n}_{n-1}\times \bm{B}^{n-1},\bm{v} ) - (p^{n},\nabla\cdot \bm{v}) 
& = \langle \bm{f},\bm{v} \rangle,\\
\label{picard2} 
 \S (\bm{j}^{n}_{n-1},\bm{F}) - \alpha ( \bm{B}^{n}, \nabla\times \bm{F}) &=0, \\
\label{picard3}
\alpha(\nabla\times \bm{E}^{n}, \bm{C}) +(r^{n}, \nabla\cdot \bm{C})&= 0, \\ 
\label{picard4} 
 -(\nabla\cdot \bm{u}^{n}, q) &=0,\\ \label{picard5}
 (\nabla\cdot \bm{B}^{n}, s) &= 0, 
\end{align}
where $\bm{j}^{n}_{n-1}$ is defined by
$\bm{j}^{n}_{n-1} = \bm{E}^{n} + \bm{u}^{n}\times \bm{B}^{n-1} $.
\end{algorithm}

The divergence-free property, compatibility and energy estimates can be obtained in an analogous way:
\begin{theorem}
For any possible solution to Algorithm \ref{alg:picard-s}:
\begin{enumerate}
\item
magnetic Gauss's law holds precisely:
$$
\nabla\cdot \bm{B}^{n}=0.
$$
\item
the Lagrange multiplier $r^{n}=0$, therefore \eqref{picard3} has the form
$$
\nabla\times \bm{E}^{n}=0.
$$
\item
the energy estimates hold:
$$
R_{e}^{-1}\|\nabla\bm{u}^{n}\|^{2}+\S\|\bm{j}^{n}_{n-1}\|^{2}=\langle \bm{f}, \bm{u}^{n}\rangle,
$$
and
\begin{align}\label{energy-E}
\frac{1}{2}R_{e}^{-1}\|\nabla\bm{u}^{n}\|^{2}+\S \|\bm{j}^{n}_{n-1}\|^{2}\leq \frac{1}{2}R_{e}\|\bm{f}\|_{-1}^{2}.
\end{align}
\end{enumerate}
\end{theorem}

We will use the Brezzi theory to prove the well-posedness of the Picard iterations. We first recast Picard iterations (Algorithm \ref{alg:picard-s}) as follows. Given $(\bm{u}^{-}, \bm{B}^{-})\in \bm{W}_{h}$. 
For ${\bm{\mathfrak{U}}}=(\bm{u,E,B})$, ${\bm{\mathfrak{V}}}=(\bm{v,F,C})\in \bm{X}_{h}$ and $(p, r), 
(q, s)\in \bm{Y}_{h}$, define bilinear forms  $\bm{a}_{s, L}(\cdot, \cdot)$ and $\bm{b}(\cdot, \cdot)$:
\begin{align*}
\bm{a}_{s, L}({\bm{\mathfrak{U}}}, {\bm{\mathfrak{V}}})
 :=&{{1 \over 2}L\left ( \bm{u}^-; \bm{u}, \bm{v} \right )+R_{e}^{-1}(\nabla \bm{u}, \nabla \bm{v}) }
+ \S (\bm{E}+\bm{u}\times {\bm{B}} ^{-}, \bm{F}+\bm{v}\times {\bm{B}} ^{-})\\
& 
- \alpha (\bm{B}, \nabla\times \bm{F})
+ \alpha (\nabla \times \bm{E}, \bm{C}).
\end{align*}

Given a nonlinear iterative step, the mixed form of the iterative scheme in Algorithm \ref{alg:picard-s}
can be written as: for any $\bm{h}=(\bm{f}, \bm{r}, \bm{l})\in \bm{X}_{h}^{\ast} $ and $\bm{g}\in \bm{Y}_{h}^*$, 
find $({\bm{\mathfrak{U}}}, \bm{x})\in  \bm{X}_{h} \times \bm{Y}_{h}$, such that for any 
$({\bm{\mathfrak{V}}}, \bm{y}) \in \bm{X}_{h} \times \bm{Y}_{h}$,  
\begin{align}
\begin{cases}
& \bm{a}_{s, L}({\bm{\mathfrak{U}}}, \bm{\eta})+ \bm{b}_{s}({\bm{\mathfrak{V}}}, \bm{x})
=\langle \bm{h}, \bm{\eta} \rangle, \\
& \bm{b}_{s}({\bm{\mathfrak{U}}},\bm{y})
=\langle \bm{g}, \bm{y} \rangle .
\end{cases}
\label{brezzip}
\end{align}

To prove the well-posedness of \eqref{brezzip} based on the Brezzi theory, we need to verify the boundedness of each term, the inf-sup condition of $\bm{b}_{s}(\cdot, \cdot)$ and the coercivity of $\bm{a}_{s, L}(\cdot,\cdot)$ on $\bm{X}_{h}^{00}$.

 For the inf-sup condition of $\bm{b}_{s}(\cdot, \cdot)$,  we have:
\begin{lemma}{\bf (inf-sup conditions of $\bm{b}_{s}(\cdot,\cdot)$)} \label{lem:inf-sup-b-s}
There exists a positive constant $\gamma$ such that
$$
\inf_{\bm{y}\in \bm{Y}_{h}}\sup_{{\bm{\mathfrak{V}}}\in \bm{X}_{h}} \frac{\bm{b}_{s}(\bm{\mathfrak{V}}, \bm{y})}{\|\bm{\mathfrak{V}}\|_{\bm{X}}\|\bm{y}\|_{\bm{Y}}}\geq \gamma>0.
$$
\end{lemma}
\begin{proof}
There exists a positive constant $\gamma_{0}>0$ such that
$$
\inf_{q\in Q_{h}}\sup_{\bm{v}\in \bm{V}_{h}} \frac{-(\nabla\cdot \bm{v}, q)}{\|\bm{v}\|_{1}\|q\|}\geq \alpha_{0}>0.
$$
Consequently, for any $q\in Q_{h}$ there exists $\bm{v}_{q}\in \bm{V}_{h}$, such that
$$
-(\nabla\cdot\bm{v}_{q}, q)\geq \gamma_{0} \|q\|^{2},
$$
and
$$
\|\bm{v}_{q}\|_{1}= \|q\|.
$$

For the magnetic multiplier, we have $\nabla\cdot {H_{0}^{h}(\mathrm{div}, \Omega)}={L_{0,h}^2(\Omega)}$. For any $s\in {L_{0,h}^2(\Omega)}$, there exists $\bm{C}_{s}\in {H_{0}^{h}(\mathrm{div}, \Omega)}$ such that $\nabla\cdot \bm{C}_{s}=s$, $\|\bm{C}_{s}\|_{\mathrm{div}}\leq C\|s\|$, where $C$ is a positive constant.

For any ${\bm{\mathfrak{V}}}=(q, s)$, take $\bm{y}=(\bm{v}_{q}, \bm{C}_{s})$.
Then
$$
\bm{b}_{s}( {\bm{\mathfrak{V}}}, \bm{y})=-(\nabla\cdot \bm{v}_{q}, q)+(\nabla\cdot \bm{C}_{s}, s)\geq \gamma_{0} \|q\|^{2}+\|{s}\|^{2}\geq  \min (\gamma_{0},1)\|\bm{y}\|_{\bm{Y}}^{2} ,
$$
and
$$
\|\bm{v}_{q}\|_{1}^{2}+\|\bm{C}_{s}\|_{\mathrm{div}}^{2}\leq \|q\|^{2}+C^{2}\|s\|^{2}\leq \max (1,C^{2}) \|\bm{y}\|_{\bm{Y}}^{2} .
$$

This completes the proof.
\end{proof}

\begin{theorem}\label{thm:wellposed-picard}
Problem \eqref{brezzip}, therefore Algorithm \ref{alg:picard-s}, is well-posed with the norms defined by \eqref{stationary-BE-norm-X} and \eqref{norm-Y}.
\end{theorem}
\begin{proof}
The boundedness of the variational form is obvious from the definition of $\|\cdot\|_{\bm{X}}$.
Moreover, we note that $\bm{a}_{s, L}({\bm{\mathfrak{U}}}, {\bm{\mathfrak{U}}})=R_{e}^{-1}\|\nabla\bm{u}\|+ \S \|\bm{E}+\bm{u}\times {\bm{B}} ^{-}\|^{2}$. Therefore the bilinear form $\bm{a}_{s, L}(\cdot,\cdot)$ is coercive on  $\bm{X}_{h}^{00}$. 

Combining the boundedness of the variational form, the inf-sup condition of $\bm{b}_{s}(\cdot, \cdot)$ (Lemma \ref{lem:inf-sup-b-s}) and the coercivity of $\bm{a}_{s, L}(\cdot,\cdot)$ on $\bm{X}_{h}^{00}$, we complete the proof.

\end{proof}

From the triangular inequality and  Hölder's inequality, we have
\begin{align*}
\|\bm{E}\|\leq \|\bm{E}+\bm{u}\times \bm{B}^{-}\|+\|\bm{u}\times \bm{B}^{-}\|\lesssim \|\bm{E}+\bm{u}\times \bm{B}^{-}\|+\|\bm{u}\|_{1}\| \bm{B}^{-}\|_{0, 3}.
\end{align*}
In  Picard iterations (Algorithm \ref{alg:picard-s}), function $\bm{B}^{-}$ is given by the magnetic field from the previous iterative step, i.e. $\bm{B}^{-}=\bm{B}^{n-1}$. We have the following estimate:
\begin{align}\label{E-bound}
\|\bm{B}^{-}\|_{0, 3}=\|\bm{B}^{n-1}\|_{0, 3}\lesssim \|\nabla_{h}\times \bm{B}^{n-1}\|\lesssim \left \|\bm{f}\right \|_{-1},
\end{align}
where the last equality is due to the energy law.

Therefore the $L^{2}$ norm of the electric field $\bm{E}$ can be bounded by $\|(\bm{u}, \bm{E}, \bm{B})\|_{\bm{X}}$ and given data, i.e., norm $\|(\bm{u}, \bm{E}, \bm{B})\|_{\bm{X}}$ is equivalent to the decoupled norm 
$$
\left ( \|\bm{u}\|_{1}^{2}+\|\bm{E}\|^{2}_{\curl}+\|\bm{B}\|_{\div}^{2}\right )^{\frac{1}{2}}.
$$
The constants involved in the equivalence depend on $\|\bm{B}^{-}\|_{0, 3}$ which further depends on $\|\bm{f}\|_{-1}$.


\subsection{Schemes without magnetic Lagrange multipliers}

Thanks to the structure-preserving properties of the discrete de Rham complex, we can design a finite element scheme for stationary MHD problems without using magnetic multipliers. The resulting scheme is equivalent to \eqref{fem-discretization}, therefore magnetic Gauss's law is precisely preserved.  

Consider the following weak form:
\begin{problem}\label{prob:continuous--no-multiplier}
Find $(\bm{u}_{h}, \bm{E}_{h}, \bm{B}_{h})\in \bm{X}_{h}$ and $ p_{h}\in Q_{h}$, such that for any
 $(\bm{v}, \bm{F}, \bm{C})\in \bm{X}_{h} $ and $ q\in Q_{h}$,
\begin{equation}\label{eqn:continuous--no-multiplier}
\left \{
\begin{aligned}
 L(\bm{u}_{h}; \bm{u}_{h}, \bm{v})
  + R_{e}^{-1} (\nabla \bm{u}_{h}, \nabla \bm{v})  
  - \S (\bm{j}_{h}\times \bm{B}_{h},\bm{v} ) - (p_{h},\nabla\cdot \bm{v}) 
 &= \langle \bm{f},\bm{v} \rangle,\\
\S(\bm{j}_{h},\bm{F}) - \alpha ( \bm{B}_{h}, \nabla\times \bm{F}) &=0, \\
 \alpha (\nabla\times \bm{E}_{h}, \bm{C})+\alpha(\nabla\cdot\bm{B}_{h}, \nabla\cdot\bm{C}) &= 0, \\ 
 -(\nabla\cdot \bm{u}_{h}, q) &=0,
\end{aligned}
\right .
\end{equation}
where $\bm{j}_{h}$ is given from Ohm's law: 
$\bm{j}_{h} = \bm{E}_{h} + \bm{u}_{h}\times \bm{B}_{h} $.
\end{problem}

Compared with Problem \ref{prob:continuous}, the magnetic Lagrange multiplier has been removed and we augment the variational formulation by introducing   $(\nabla\cdot\bm{B}_{h}, \nabla\cdot\bm{C})$ term. Next we verify some properties of the proposed schemes. 
\begin{theorem}
Any solution to Problem \ref{prob:continuous--no-multiplier} satisfies
\begin{enumerate}
\item magnetic Gauss's law in the strong sense:
$$
\nabla\cdot\bm{B}_{h}=0,
$$
\item
the discrete energy law: 
$$
{R_{e}^{-1}}\|\nabla\bm{u}_{h}\|^{2}+\S\|\bm{j}_{h}\|^{2}=\langle \bm{f}, \bm{u}_{h}\rangle,
$$
and
$$
\frac{1}{2}R_{e}^{-1}\|\nabla\bm{u}_{h}\|^{2}+\S\|\bm{j}_{h}\|^{2}\leq \frac{R_{e}}{2}\|\bm{f}\|_{-1}^{2}.
$$
\end{enumerate}
\end{theorem}
\begin{proof}
The proof of the discrete energy law is almost the same as Problem
 \ref{prob:continuous}. Therefore we only prove the magnetic Gauss's law. 

Taking $\bm{C}=\nabla\times \bm{E}_{h}$  in \eqref{eqn:continuous--no-multiplier},  we have $\nabla\times \bm{E}_{h}=0$. Therefore 
$$
(\nabla\cdot\bm{B}_{h}, \nabla\cdot\bm{C}_{h})=0, \quad\forall \bm{C}_{h}\in H^{h}_{0}(\div 0, \Omega).
$$
This implies that $\nabla\cdot\bm{B}_{h}=0$.

\end{proof}

To verify the well-posedness, we can formally eliminate  $\bm{E}_{h}$  to get a system with $\bm{u}_{h}$, $p_{h}$ and $\bm{B}$. For the  Lagrange multiplier $p_{h}$, one can verify the inf-sup condition of the  $(\nabla\cdot\bm{u}, q)$ pair.  
We can also verify the boundedness and coercivity in $\bm{V}_{h}^{0}\times H^{h}_{0}(\curl, \Omega)\times H^{h}_{0}(\div, \Omega)$ for other terms. Consequently, we have the well-posedness result:
\begin{theorem}
Problem \ref{prob:continuous--no-multiplier} has at least one solution $(\bm{u}_{h},  \bm{E}_{h}, \bm{B}_{h}, p_{h})\in \bm{X}_{h}\times Q_{h}$. With suitable data, the solution is unique. 
\end{theorem}

We can similarly define Picard iterations: For $n=1, 2, \cdots$, given $\left (\bm{u}^{n-1}, \bm{B}^{n-1}\right )\in \bm{W}_{h}$,  find $(\bm{u}^{n}, \bm{E}^{n}, \bm{B}^{n})\in \bm{X}_{h}$ and $ p^{n}\in Q_{h}$, such that for any
 $(\bm{v}, \bm{F}, \bm{C})\in \bm{X}_{h} $ and $ q\in Q_{h}$,
\begin{equation}\label{eqn:continuous--no-multiplier-picard}
\left \{
\begin{aligned}
 L(\bm{u}^{n-1}; \bm{u}^{n}, \bm{v})
  + R_{e}^{-1} (\nabla \bm{u}^{n}, \nabla \bm{v})  
  - \S (\bm{j}^{n}_{n-1}\times \bm{B}^{n-1},\bm{v} ) - (p^{n},\nabla\cdot \bm{v}) 
 &= \langle \bm{f},\bm{v} \rangle,\\
\S(\bm{j}^{n}_{n-1},\bm{F}) - \alpha ( \bm{B}^{n}, \nabla\times \bm{F}) &=0, \\
 \alpha (\nabla\times \bm{E}^{n}, \bm{C})+\alpha(\nabla\cdot\bm{B}^{n}, \nabla\cdot\bm{C}) &= 0, \\ 
 -(\nabla\cdot \bm{u}^{n}, q) &=0,
\end{aligned}
\right .
\end{equation}
where $\bm{j}^{n}_{n-1}$ is given by  Ohm's law: 
$\bm{j}^{n}_{n-1} = \bm{E}^{n} + \bm{u}^{n}\times \bm{B}^{n-1} $.
one can similarly verify the following properties:
\begin{theorem}
Any solution to Problem \ref{eqn:continuous--no-multiplier-picard} satisfies
\begin{enumerate}
\item magnetic Gauss's law in the strong sense:
$$
\nabla\cdot\bm{B}^{n}=0,\quad n=1, 2, \cdots,
$$
\item
the discrete energy law: 
$$
{R_{e}^{-1}}\|\nabla\bm{u}^{n}\|^{2}+\S\|\bm{j}^{n}_{n-1}\|^{2}=\langle \bm{f}, \bm{u}^{n}\rangle,
$$
and
$$
\frac{1}{2}R_{e}^{-1}\|\nabla\bm{u}^{n}\|^{2}+\S\|\bm{j}^{n}_{n-1}\|^{2}\leq \frac{R_{e}}{2}\|\bm{f}\|_{-1}^{2}.
$$
\end{enumerate}
\end{theorem}
Analogous to Theorem \ref{thm:wellposed-stationary}, we can verify the well-posedness:
\begin{theorem}
Variational form \eqref{eqn:continuous--no-multiplier-picard} has a unique solution $(\bm{u}^{n}, \bm{E}^{n},  \bm{B}^{n}, p^{n})\in \bm{X}_{h}\times Q_{h}$. 
\end{theorem}

\section{Convergence of Picard iterations}

\begin{theorem}
\label{thm_picard_conv}
If both $R_{e}^{2}\Vert \bm{f} \Vert_{-1}$ and $R_{e} R_{m}^{\frac{3}{2}} \Vert \bm{f} \Vert_{-1}$  are small enough, 
then the method \eqref{fem-discretization} (Problem~\ref{prob:continuous}) with the boundary condition (\ref{boundary_cond}) 
has a unique solution, and the solution of the Picard 
iteration (Algorithm~\ref{alg:picard-s}) converges to it with respect to 
the norms defined by \eqref{stationary-BE-norm-X} and \eqref{norm-Y}.
\end{theorem}

We skip the proof of Theorem~\ref{thm_picard_conv}, since it is a simpler version of the proofs of the following Theorem~\ref{energy_est}.

\section{Convergence of finite element methods}
\label{sec_conv}
In this section, we present the error estimates of the method \eqref{fem-discretization}, which is for 
the boundary condition (\ref{boundary_cond}). 
Our analysis is based on the minimum regularity assumption on the exact solutions (c.f. \cite{Schotzau.D.2004a}). Namely, we assume
\begin{equation}\label{regularity}
 \bm{u} \in [H^{1 + \sigma}(\Omega)]^3, \quad \bm{B}, \nabla \times \bm{B}, \bm{E} \in [H^{\sigma}(\Omega)]^3, \quad p 
  \in H^{\sigma}(\Omega) \cap L^2_0(\Omega),
\end{equation}
here $\sigma > \frac12$. Next we introduce notations used in the analysis. For a generic unknown $\mathcal{U}$ and its numerical counterpart $\mathcal{U}_h$ we split the error as:
\[
\mathcal{U} - \mathcal{U}_h = (\mathcal{U} - \Pi \mathcal{U}) + (\Pi \mathcal{U} - \mathcal{U}_h) := \delta_{U} + e_U.
\]
Here $\Pi \mathcal{U}$ is a projection of $\mathcal{U}$ into the corresponding discrete space that $\mathcal{U}_h$ belongs to. Namely, for $(\bm{E}, r)$ we use the projections $(\Pi^{\mathrm{curl}} \bm{E}, \Pi^0 r)$ in the commuting diagram in Figure \ref{exact-sequence}. For $\bm{B}$ and $p$ we define the $L^2$ projection $\Pi^D \bm{B}, \Pi^Q p$ into $H^h_0(\div0, \Omega), Q_h$ respectively. Notice here $r = 0$ implies that $\Pi^0 r = 0$ and hence $\delta_r = 0$.  Finally, for the velocity $\bm{u}$ we define $(\Pi^V \bm{u}, \tilde{p}_h) \in \bm{V}_h \times Q_h$ be the unique numerical solution of the Stokes equation:
\begin{align}\label{proj_u1}
(\nabla \Pi^V {\bm{u}}, \nabla \bm{v}) + (\tilde{p}_{h}, \nabla \cdot \bm{v}) &= (\nabla \bm{u}, \nabla \bm{v}), \\
\label{proj_u2}
(\nabla \cdot \Pi^V \bm{u}, q) & = 0,
\end{align} 
for all $(\bm{v}, q) \in \bm{V}_h \times Q_h$. Notice that $(\bm{u}, 0)$ is the exact solution of the Stokes equations:
\begin{align*}
-\Delta \tilde{\bm{u}} + \nabla \tilde{p} & = - \Delta \bm{u}, \\
\nabla \cdot \tilde{\bm{u}} & = 0,
\end{align*}
with $\tilde{\bm{u}} = \bm{0}$ on $\partial\Omega$. 
Hence, if $\bm{V}_h \times Q_h$ is a stable Stokes pair, we should have optimal approximation for the above equation:
\begin{equation}\label{proj_u}
\|\bm{u} - \Pi^V \bm{u}\|_1 \le C \inf_{\bm{v}\in \bm{V}_{h}} \Vert \bm{u} - \bm{v}\Vert_{1}.
\end{equation}
Immediately we can see that 
\begin{equation}\label{proj_u3}
(\delta_{\bm{u}}, q) = 0 \quad \text{for all} \quad q \in Q_h.
\end{equation}

Since $\bm{B}, \Pi^D \bm{B}, \bm{B}_h \in H_0(\div0, \Omega)$ and $\bm{E}, \bm{E}_h, \Pi^{\mathrm{curl}} \bm{E} \in H_0(\curl0, \Omega)$ we have
\begin{equation}\label{div0_curl0}
\nabla \cdot e_{\bm{B}} = \nabla \cdot \delta_{\bm{B}} = 0, \quad \nabla \times e_{\bm{E}} = \nabla \times \delta_{\bm{E}} = 0.
\end{equation}
In addition, since $\nabla \times H^{h}_0(\curl, \Omega) \subset H_0^h(\div0, \Omega)$ we have
\begin{equation}\label{proj_B1}
(\delta_{\bm{B}}, \nabla \times \bm{F}) = 0 \quad \text{for all} \quad \bm{F} \in H_0^h(\curl, \Omega). 
\end{equation}
Let $\Pi^{\mathrm{div}}$ be the $H(\div)$-conforming projection in the commuting diagram 
in Figure \ref{exact-sequence}. Obviously, $\Pi^{\mathrm{div}}\bm{B} \in H^h_0(\div0, \Omega)$. 
Then, due to the construction of $\Pi^{D}$, we have 
\begin{align}
\label{proj_B2}
\Vert\Pi^{D} \bm{B} - \bm{B} \Vert = \inf_{\bm{C}\in H^h_0(\div0, \Omega)} \Vert \bm{B} - \bm{C} \Vert 
\leq \Vert \Pi^{\mathrm{div}}\bm{B} - \bm{B} \Vert 
\leq C \inf_{\bm{C}\in H^{h}_{0}(\div, \Omega)} \Vert \bm{B} - \bm{C}\Vert.
\end{align}

Now we are ready to present the error equations for the error estimates. Notice that the exact solution $(\bm{u}, \bm{E}, \bm{B}, r, p)$ also satisfies the discrete formulation \eqref{fem-discretization}. Subtracting two systems, with the spliting of the errors and above properties of the projections \eqref{proj_u3}, \eqref{div0_curl0} and \eqref{proj_B1}, we arrive at:
\begin{align}
\nonumber
 (L(\bm{u}; \bm{u}, \bm{v}) - L(\bm{u}_h; \bm{u_h}, \bm{v})) 
  + R_{e}^{-1} (\nabla e_{\bm{u}}, \nabla \bm{v})  
  &- \S (\bm{j} \times \bm{B} - \bm{j}_h \times \bm{B}_h,\bm{v} )  - (e_p,\nabla\cdot \bm{v})  \\
 \label{err-1}
&= - R_{e}^{-1} (\nabla \delta_{\bm{u}}, \nabla \bm{v}) + (\delta_p,\nabla\cdot \bm{v}),\\
\label{err-2}
\S(\bm{j} - \bm{j}_h,\bm{F}) - \alpha ( e_{\bm{B}}, \nabla\times \bm{F}) &= 0, \\
\label{err-3}
 \alpha (\nabla\times e_{\bm{E}}, \bm{C}) +(e_r, \nabla\cdot \bm{C}) &= - (\delta_r, \nabla\cdot \bm{C}), \\
 \label{err-4}
 -(\nabla\cdot e_{\bm{u}}, q) &= 0,\\
\label{err-5}
 (\nabla\cdot e_{\bm{B}}, s) &= 0, 
\end{align}
for all $(\bm{v}, \bm{F}, \bm{C})\in \bm{X}_{h} $ and $ (q, s)\in \bm{Y}_{h}$.

\begin{lemma}\label{energy_id}
We have the energy identity:
\begin{align}
\nonumber
R_e^{-1}\|\nabla e_{\bm{u}}\|^2 + \alpha \|\nabla_h \times e_{\bm{B}}\|^2 = & - (L(\bm{u}; \bm{u}, e_{\bm{u}}) - L(\bm{u}_h; \bm{u_h}, e_{\bm{u}})) + (\delta_p, \nabla \cdot e_{\bm{u}}) - R^{-1}_e (\nabla \delta_{\bm{u}}, \nabla e_{\bm{u}})  \\
\nonumber
&+ \S (\bm{j} \times \bm{B} - \bm{j}_h \times \bm{B}_h,e_{\bm{u}})   + \S(\bm{j} - \bm{j}_h, \nabla_h \times e_{\bm{B}}).
\end{align}
\end{lemma}
\begin{proof}
Taking $\bm{v} = e_{\bm{u}}, \bm{F} = -\nabla_h \times e_{\bm{B}}, q = e_p$ in \eqref{err-1}, \eqref{err-2} and \eqref{err-4} and adding these equations, we can obtain the above identity by rearranging terms in the equation. 
\end{proof}
 
From the above result we can see that it suffices to bound the terms on the right hand side of the energy identity to get the error estimates in the energy norm. The first four terms can be handled with standard tools for Navier-Stokes equations, see \cite{Girault.V;Raviart.P.1986a,ZhangHeYang} for instance. In particular, we need the following continuity result for the advection term, see \cite{ZhangHeYang}: 

\begin{lemma}\label{convection}
For any $\bm{u}, \bm{v}, \bm{w} \in [H^1_0(\Omega)]^3$, we have
\[
L(\bm{w}; \bm{u}, \bm{v}) \le C \|\nabla \bm{w}\| \|\nabla \bm{u}\| \|\nabla \bm{v}\|, 
\]
where $C$ solely depends on the domain $\Omega$.
\end{lemma}

In order to bound the last two terms, we need the following auxiliary results:

\begin{lemma}\label{aux1}
If the regularity assumption \eqref{regularity} is satisfied, we have
\begin{align*}
\Vert \bm{u} \times \bm{B} - \bm{u}_h \times \bm{B}_h \Vert & \leq 
C \big( \Vert \bm{u}\Vert_{0,\infty} \Vert \delta_{\bm{B}}\Vert 
+ R_{e} \Vert \bm{f}\Vert_{-1} \Vert \nabla_{h}\times e_{\bm{B}} \Vert 
+ R_{e}^{\frac{1}{2}} R_{m} \S^{-\frac{1}{2}} \Vert \bm{f} \Vert_{-1}
(\Vert e_{\bm{u}}\Vert_{1} + \Vert \delta_{\bm{u}}\Vert_{1}) \big),\\ 
\Vert e_{\bm{E}}\Vert & \leq \Vert \delta_{\bm{E}}\Vert + \Vert \bm{u} \times \bm{B} - \bm{u}_h \times \bm{B}_h \Vert.
\end{align*}
\end{lemma}
 \begin{proof}
For $\|\bm{u} \times \bm{B} - \bm{u}_h \times \bm{B}_h\|$, we have
\begin{align}\nonumber
\|\bm{u} \times \bm{B} - \bm{u}_h \times \bm{B}_h\| &= \|\bm{u} \times \delta_{\bm{B}} + \bm{u} \times e_{\bm{B}} +(\delta_{\bm{u}} + e_{\bm{u}}) \times \bm{B}_h\| \\
\nonumber
& \le \|\bm{u} \times \delta_{\bm{B}}\| + \|\bm{u} \times e_{\bm{B}}\| + \|(\delta_{\bm{u}} + e_{\bm{u}}) \times \bm{B}_h\| \\
\nonumber
& \le \|\bm{u}\|_{0,\infty} \|\delta_{\bm{B}}\| + \|\bm{u}\|_{0,6} \|e_{\bm{B}}\|_{0,3} 
+ (\|\delta_{\bm{u}}\|_{0,6} + \|e_{\bm{u}}\|_{0,6}) \|\bm{B}_h\|_{0,3},
\end{align}
the last step is due to H\"{o}lder's inequality. 
By 
\eqref{def:C1} and Theorem \ref{L3estimate}, we have
\[
\|\bm{u} \times \bm{B} - \bm{u}_h \times \bm{B}_h\|
 \le C (\|\bm{u}\|_{0,\infty} \|\delta_{\bm{B}}\| + \|\bm{u}\|_{1} \|\nabla_h \times e_{\bm{B}}\|
 + (\|\nabla \delta_{\bm{u}}\| + \|\nabla e_{\bm{u}}\|) \|\nabla_h \times \bm{B}_h\|).
\]
Finally we can obtain the estimate for this term by the stability result in Theorem \ref{stability_fem} and Remark \ref{stab_const}. Next, taking $\bm{F} = e_{\bm{E}}$ in \eqref{err-2}, by \eqref{div0_curl0}, we have
\[
(\bm{j} - \bm{j}_h, e_{\bm{E}}) = 0.
\]
By the definition of $\bm{j}, \bm{j}_h$, we obtain:
\[
\|e_{\bm{E}}\|^2 = - (\delta_{\bm{E}}, e_{\bm{E}}) - (\bm{u} \times \bm{B} - \bm{u}_h \times \bm{B}_h, e_{\bm{E}}).
\]
The proof is completed by Cauchy-Schwarz inequality. 
 \end{proof}

Now we are ready to give our first error estimate:

\begin{theorem}\label{energy_est}
If the regulartity assumtion \eqref{regularity} holds, in addition, both $R_{e}^{2}\Vert \bm{f} \Vert_{-1}$ 
and $R_{e} R_{m}^{\frac{3}{2}} \Vert \bm{f} \Vert_{-1}$  are small enough, then we have
\[
R^{-\frac12}_e \|\nabla e_{\bm{u}}\| + \alpha^{\frac{1}{2}}\Vert e_{\bm{B}} \Vert_{0,3}
+ \alpha^{\frac12} \|\nabla_h \times e_{\bm{B}}\|
 \le \mathbb{C} (\|\delta_p\| + \|\nabla \delta_{\bm{u}}\| + (\|\bm{u}\|_{1 + \sigma} + \|\nabla \times \bm{B}\|_{\sigma})\|\delta_{\bm{B}}\| + \|\delta_{\bm{E}}\|),
\]
where $\mathbb{C}$ depends on all the parameters $R_m, R_e, \S$ and $\|\bm{f}\|_{-1}$.
\end{theorem}

\begin{proof}
Since $\nabla\cdot e_{\bm{B}} = 0$ by (\ref{div0_curl0}), we can apply Theorem~\ref{L3estimate} to obtain 
\begin{align*}
\Vert e_{\bm{B}}\Vert_{0,3} \leq C \Vert \nabla_{h} \times e_{\bm{B}} \Vert.
\end{align*}

By Lemma \ref{energy_id}, it suffices to bound terms on the right hand side in the energy identity. The two bilinear terms can be bounded by using Cauchy-Schwarz inequality as,
\begin{align*}
(\delta_p, \nabla \cdot e_{\bm{u}}) &\le \|\delta_p\| \|\nabla e_{\bm{u}}\|, \\
R^{-1}_e (\nabla \delta_{\bm{u}}, \nabla e_{\bm{u}}) & \le R^{-1}_e \|\nabla \delta_{\bm{u}}\| \|\nabla e_{\bm{u}}\|.
\end{align*}
For the convection term, by Lemma \ref{convection} we have
\begin{align*}
L(\bm{u}; \bm{u}, e_{\bm{u}}) - L(\bm{u}_h; \bm{u_h}, e_{\bm{u}}) &= L(\bm{u} - \bm{u}_h; \bm{u}, e_{\bm{u}}) + L(\bm{u}_h; \bm{u} - \bm{u}_h, e_{\bm{u}}) \\
& \le  C(\|\nabla \delta_{\bm{u}}\| + \|\nabla e_{\bm{u}}\|) \|\nabla \bm{u}\| \|\nabla e_{\bm{u}}\| + C (\|\nabla \delta_{\bm{u}}\| + \|\nabla e_{\bm{u}}\|) \|\nabla \bm{u}_h\| \|\nabla e_{\bm{u}}\| \\
& \leq C R_{e}^{2}\Vert \bm{f}\Vert_{-1} \big(  R_{e}^{-1} \Vert \nabla e_{\bm{u}}\Vert^{2}
+ R_{e}^{-1}\Vert \nabla \delta_{\bm{u}} \Vert \cdot \Vert \nabla e_{\bm{u}} \Vert \big),
\end{align*} 
the last step is by the stability result \eqref{stab_exact_solution} in Remark \ref{stab_const}. 
In order to obtain the convergent result, we need $R_{e}^{2}\Vert \bm{f} \Vert_{-1}$ to be small enough.

Next we need to bound the last two terms in Lemma \ref{energy_id}. By Cauchy-Schwarz inequality we have
\begin{align*}
\S(\bm{j} - \bm{j}_h, \nabla_h \times e_{\bm{B}}) & \le \S \|\bm{j} - \bm{j}_h\| \|\nabla_h \times e_{\bm{B}}\| \\
&= \S \|\bm{E} + \bm{u} \times \bm{B} - (\bm{E}_h + \bm{u}_h \times \bm{B}_h)\| \|\nabla_h \times e_{\bm{B}}\| \\
& \le \S (\|\delta_{\bm{E}}\| + \|e_{\bm{E}}\| + \|\bm{u} \times \bm{B} - \bm{u}_h \times \bm{B}_h\|) \|\nabla_h \times e_{\bm{B}}\| \\
& \leq C \S \big( \Vert \delta_{\bm{E}} \Vert 
+ \Vert \bm{u}\Vert_{0,\infty} \Vert \delta_{\bm{B}}\Vert 
+ R_{e} \Vert \bm{f}\Vert_{-1} \Vert \nabla_{h}\times e_{\bm{B}} \Vert \\ 
& \qquad  + R_{e}^{\frac{1}{2}} R_{m} \S^{-\frac{1}{2}} \Vert \bm{f} \Vert_{-1}
(\Vert e_{\bm{u}}\Vert_{1} + \Vert \delta_{\bm{u}}\Vert_{1})
\big) \Vert \nabla_{h}\times e_{\bm{B}} \Vert \\
& = C \S \big( \Vert \delta_{\bm{E}} \Vert 
+ \Vert \bm{u}\Vert_{0,\infty} \Vert \delta_{\bm{B}}\Vert 
+ R_{e}^{\frac{1}{2}} R_{m} \S^{-\frac{1}{2}} \Vert \bm{f} \Vert_{-1}
\Vert \delta_{\bm{u}}\Vert_{1}\big) \Vert \nabla_{h}\times e_{\bm{B}} \Vert \\
& \qquad  + C \big(R_{e}\S \Vert \bm{f}\Vert_{-1} \Vert \nabla_{h}\times e_{\bm{B}} \Vert^{2} 
+ R_{e}^{\frac{1}{2}}R_{m}\S^{\frac{1}{2}} \Vert \bm{f} \Vert_{-1} 
\Vert e_{\bm{u}} \Vert_{1} \Vert \nabla_{h}\times e_{B} \Vert \big) \\
& \leq C \S \big( \Vert \delta_{\bm{E}} \Vert 
+ \Vert \bm{u}\Vert_{0,\infty} \Vert \delta_{\bm{B}}\Vert 
+ R_{e}^{\frac{1}{2}} R_{m} \S^{-\frac{1}{2}} \Vert \bm{f} \Vert_{-1}
\Vert \delta_{\bm{u}}\Vert_{1}\big) \Vert \nabla_{h}\times e_{\bm{B}} \Vert \\
& \qquad + C \big(R_{e}R_{m} \Vert \bm{f}\Vert_{-1} (\alpha \|\nabla_h \times e_{\bm{B}}\|^{2})
+R_{e}R_{m}^{\frac{3}{2}} \Vert \bm{f} \Vert_{-1}
(R^{-1}_e \|\nabla e_{\bm{u}}\|^{2} + \alpha \|\nabla_h \times e_{\bm{B}}\|^{2})  \big).
\end{align*}
In order to obtain the convergent result, we need 
$R_{e} R_{m}^{\frac{3}{2}} \Vert \bm{f} \Vert_{-1}$ to be small enough.
Finally, for the last term we begin by spliting the term into three terms and applying the generalized H\"{o}lder's inequality to have
\begin{align*}
\S(\bm{j} \times \bm{B} - \bm{j}_h \times \bm{B}_h,e_{\bm{u}}) & = \S (\bm{j} \times \delta_{\bm{B}}, e_{\bm{u}}) + \S(\bm{j} \times e_{\bm{B}}, e_{\bm{u}}) + \S(( \bm{j}- \bm{j}_h) \times \bm{B}_h,e_{\bm{u}}) \\
& = T_1 + T_2 + T_3. 
\end{align*}
By the fact that $\bm{j} = R^{-1}_m \nabla \times \bm{B}$, we can further apply the generalized H\"{o}lder's inequalities, Sobolev embedding inequalities, $H^{\sigma}(\Omega) \hookrightarrow L^3(\Omega)$, \eqref{def:C1} and Theorem \ref{L3estimate} for $T_1$, $T_{2}$ and  $T_3$ as:
\begin{align*}
T_1 &\le \S  \|\bm{j}\|_{0,3} \|\delta_{\bm{B}}\| \|e_{\bm{u}}\|_{0,6} \le C \S R^{-1}_m \|\nabla \times \bm{B}\|_{\sigma} \|\delta_{\bm{B}}\| \|\nabla e_{\bm{u}}\|, \\
T_2 &\le \S  \|\bm{j}\| \|e_{\bm{B}}\|_{0,3} \|e_{\bm{u}}\|_{0,6} \le C \S R^{-1}_m \|\nabla \times \bm{B}\| \|\nabla_h \times e_{\bm{B}}\| \|\nabla e_{\bm{u}}\| 
\le C R_{e}^{\frac{1}{2}} \S^{\frac{1}{2}} \|\bm{f}\|_{-1} \|\nabla_h \times e_{\bm{B}}\| \|\nabla e_{\bm{u}}\| \\
& \leq C R_{e} R_{m}^{\frac{1}{2}} \Vert \bm{f} \Vert_{-1}
 (R^{-1}_e \|\nabla e_{\bm{u}}\|^{2} + \alpha \|\nabla_h \times e_{\bm{B}}\|^{2}),\\
T_3 &\le \S \|\bm{j} - \bm{j}_h\| \|\bm{B}_h\|_{0,3} \|e_{\bm{u}}\|_{0,6} 
\le C \S  \|\bm{j} - \bm{j}_h\| \|\nabla_h \times \bm{B}_h\| \|\nabla e_{\bm{u}}\| 
\le C R_{e}^{\frac{1}{2}} R_{m} \S^{\frac{1}{2}}  \|\bm{f}\|_{-1} \|\bm{j} - \bm{j}_h\| \|\nabla e_{\bm{u}}\| \\ 
& \le C R_{e}^{\frac{1}{2}} R_{m} \S^{\frac{1}{2}} \|\bm{f}\|_{-1} \big( \Vert \delta_{\bm{E}} \Vert 
+ \Vert \bm{u}\Vert_{0,\infty} \Vert \delta_{\bm{B}}\Vert 
+ R_{e} \Vert \bm{f}\Vert_{-1} \Vert \nabla_{h}\times e_{\bm{B}} \Vert 
 + R_{e}^{\frac{1}{2}} R_{m} \S^{-\frac{1}{2}} \Vert \bm{f} \Vert_{-1}
(\Vert e_{\bm{u}}\Vert_{1} + \Vert \delta_{\bm{u}}\Vert_{1})
\big) \|\nabla e_{\bm{u}}\| \\
& \leq C R_{e}^{\frac{1}{2}} R_{m} \S^{\frac{1}{2}} \|\bm{f}\|_{-1} 
\big( \Vert \delta_{\bm{E}} \Vert + \Vert \bm{u}\Vert_{0,\infty} \Vert \delta_{\bm{B}}\Vert 
+ R_{e}^{\frac{1}{2}} R_{m} \S^{-\frac{1}{2}} \Vert \bm{f} \Vert_{-1} \Vert \delta_{\bm{u}} \Vert_{1} \big) \\
& \qquad + C R_{e}^{\frac{3}{2}} R_{m} \S^{\frac{1}{2}} \|\bm{f}\|_{-1}^{2}  
\Vert e_{\bm{u}}\Vert_{1} \Vert \nabla_{h}\times e_{\bm{B}} \Vert 
+ C R_{e}R_{m}^{2}\Vert \bm{f}\Vert_{-1}^{2} \Vert e_{\bm{u}}\Vert_{1}^{2} \\
& \leq C R_{e}^{\frac{1}{2}} R_{m} \S^{\frac{1}{2}} \|\bm{f}\|_{-1} 
\big( \Vert \delta_{\bm{E}} \Vert + \Vert \bm{u}\Vert_{0,\infty} \Vert \delta_{\bm{B}}\Vert 
+ R_{e}^{\frac{1}{2}} R_{m} \S^{-\frac{1}{2}} \Vert \bm{f} \Vert_{-1} \Vert \delta_{\bm{u}} \Vert_{1} \big) \\
& \qquad +R_{e} R_{m}^{\frac{3}{2}} \Vert \bm{f}\Vert_{-1}^{2} 
(R^{-1}_e \|\nabla e_{\bm{u}}\|^{2} + \alpha \|\nabla_h \times e_{\bm{B}}\|^{2}) 
+ C R_{e}^{2}R_{m}^{2}\Vert \bm{f}\Vert_{-1}^{2} (R^{-1}_e \|\nabla e_{\bm{u}}\|^{2}).
\end{align*}
Referring to $T_{2}$ and $T_{3}$, we need $R_{e} R_{m}^{\frac{1}{2}} \Vert \bm{f} \Vert_{-1}$, 
$R_{e} R_{m}^{\frac{3}{2}} \Vert \bm{f}\Vert_{-1}^{2}$ and 
$R_{e}^{2}R_{m}^{2}\Vert \bm{f}\Vert_{-1}^{2}$ to be small enough 
such that convergent results can be obtained.

So, if $R_{e}^{2}\Vert \bm{f} \Vert_{-1}$ and $R_{e} R_{m}^{\frac{3}{2}} \Vert \bm{f} \Vert_{-1}$ 
are both small enough, we have
\[
R^{-\frac12}_e \|\nabla e_{\bm{u}}\| + \alpha^{\frac12} \|\nabla_h \times e_{\bm{B}}\| \le \mathbb{C} (\|\delta_p\| + \|\nabla \delta_{\bm{u}}\| + (\|\bm{u}\|_{1 + \sigma} + \|\nabla \times \bm{B}\|_{\sigma})\|\delta_{\bm{B}}\| + \|\delta_{\bm{E}}\|).
\]
Here $\mathbb{C}$ depends on all the parameters $R_m, R_e, s$ and $\|\bm{f}\|_{-1}$. This completes the proof.
\end{proof}

\section{Nonlinear scheme for the alternative boundary condition}

We propose the following variational form for  \eqref{eq:MHD-stationary} with boundary condition (\ref{boundary_cond2}):
\begin{problem}\label{prob_bc2}
Find $(\bm{u}_h, \bm{E}_h, \bm{B}_h)\in \tilde{\bm{X}}_{h}$ and $ (p_h, r_h)\in \bm{Y}_{h}$, 
such that for any $(\bm{v}, \bm{F}, \bm{C})\in \tilde{\bm{X}}_{h} $ and $ (q, s)\in \bm{Y}_{h}$,
\begin{subequations}\label{fem-discretization_bc2}
\begin{equation}\label{fem-1_bc2}
 L(\bm{u}_h; \bm{u}_h, \bm{v})
  + R_{e}^{-1} (\nabla \bm{u}_h, \nabla \bm{v})  
  - \S (\bm{j}_h\times \bm{B}_h,\bm{v} ) - (p_h,\nabla\cdot \bm{v}) 
 = \langle \bm{f},\bm{v} \rangle,\\
 \end{equation}
 \begin{equation}\label{fem-2_bc2}
\S(\bm{j}_h,\bm{F}) - \alpha ( \bm{B}_h, \nabla\times \bm{F}) =0, 
\end{equation}
 \begin{equation}\label{fem-3_bc2}
 \alpha (\nabla\times \bm{E}_h, \bm{C}) +(r_h, \nabla\cdot \bm{C})= 0, 
 \end{equation}
 \begin{equation}\label{fem-4_bc2}
 -(\nabla\cdot \bm{u}_h, q) =0,
 \end{equation}
 \begin{equation}\label{fem-5_bc2}
 (\nabla\cdot \bm{B}_h, s) = 0, 
 \end{equation}
\end{subequations}
where $\bm{j}_h$ is given by  Ohm's law: 
$\bm{j}_h = \bm{E}_h + \bm{u}_h \times \bm{B}_h$ and $r_h$ is the Lagrange multiplier which approximates $r = 0$, 
and $\tilde{\bm{X}}_{h}= \bm{V}_{h}\times H^{h}(\mathrm{curl}, \Omega)
\times H^{h}(\mathrm{div},\Omega)$.
\end{problem}

Similar to Theorem~\ref{stability_fem}, we have Theorem~\ref{stability_fem_bc2}, whose 
proof is the same as that of Theorem~\ref{stability_fem}.

\begin{theorem}\label{stability_fem_bc2}
Any solution for  Problem \ref{prob_bc2} satisfies
\begin{enumerate}
\item
magnetic Gauss's law:
$$
\nabla\cdot \bm{B}_h=0.
$$
\item  Lagrange multiplier  $r=0$, and the strong form
$$
\nabla\times \bm{E}_h =0,
$$
\item energy estimates:
\begin{align}\label{energy-1_bc2}
{R_{e}^{-1}}\|\nabla\bm{u}_h\|^{2}+\S\|\bm{j}_h\|^{2} &= \langle \bm{f}, \bm{u}_h\rangle, \\
\label{energy-2_bc2}
\frac{1}{2}R_{e}^{-1}\|\nabla\bm{u}_h\|^{2}+\S\|\bm{j}_h\|^{2} &\leq \frac{R_{e}}{2}\|\bm{f}\|_{-1}^{2}, \\
\label{energy-3_bc2}
R_m^{-1} \|\nabla_h \times \bm{B}_h\| &\le  \|\bm{j}_h\|, \\
\label{energy-4_bc2}
\Vert \nabla_h \times \bm{B}_h \Vert & \le C R_{e}^{\frac{1}{2}} R_{m} \S^{-\frac{1}{2}} \Vert \bm{f} \Vert_{-1}, \\
\label{energy-5_bc2}
\|\bm{E}_h\| &\le C R_e^{\frac32} R_m \S^{-\frac{1}{2}} \|\bm{f}\|_{-1}.
\end{align}
\end{enumerate}
\end{theorem}

Similar to the argument in Section~\ref{sec_scheme}, we can conclude that Problem~\ref{prob_bc2} is well-posed.

We define $e_{\bm{u}}$, $\delta_{\bm{u}}$, $e_{p}$, $\delta_{p}$, $e_{r}$, $\delta_{r}$ the same 
as those in Section~\ref{sec_conv}. We use $\Pi^{\curl}$ in Figure~\ref{exact-sequence-nobc} 
for the electric field $\bm{E}$. We define $e_{\bm{E}} = \Pi^{\curl} \bm{E} - \bm{E}_{h}$ 
and $\delta_{\bm{E}} = \bm{E} - \Pi^{\curl}\bm{E}$. 
For the magnetic field $\bm{B}$, we define the $L^{2}$-projection 
$\Pi^{\tilde{D}}$ into $H^{h}(\div0, \Omega)$. 
We denote  $e_{\bm{B}} = \Pi^{\tilde{D}}\bm{B} - \bm{B}_{h}$ and 
$\delta_{\bm{B}} = \bm{B} - \Pi^{\tilde{D}}\bm{B}$. 
It is easy to see that 
\begin{align*}
&\nabla\cdot e_{\bm{B}} = 0,\\
& (\bm{B} - \Pi^{\tilde{D}}\bm{B}, \nabla\times \bm{F})=0 \quad \text{ for all } F \in H^{h}(\curl, \Omega), \\
& \Vert  \bm{B} - \Pi^{\tilde{D}}\bm{B}\Vert \leq C \inf_{\bm{C}\in H^{h}(\div,\Omega)} \Vert \bm{B} - \bm{C} \Vert.
\end{align*}
Thus by using Theorem~\ref{L3estimate_new} to replace Theorem~\ref{L3estimate}, we can 
use the same argument in Section~\ref{sec_conv} to obtain Theorem~\ref{energy_est_bc2}.

\begin{theorem}\label{energy_est_bc2}
If the regulartity assumtion \eqref{regularity} holds, in addition, both $R_{e}^{2}\Vert \bm{f} \Vert_{-1}$ 
and $R_{e} R_{m}^{\frac{3}{2}} \Vert \bm{f} \Vert_{-1}$  are small enough, then we have
\[
R^{-\frac12}_e \|\nabla e_{\bm{u}}\| + \alpha^{\frac{1}{2}}\Vert e_{\bm{B}} \Vert_{0,3}
+ \alpha^{\frac12} \|\nabla_h \times e_{\bm{B}}\|
 \le \mathbb{C} (\|\delta_p\| + \|\nabla \delta_{\bm{u}}\| + (\|\bm{u}\|_{1 + \sigma} + \|\nabla \times \bm{B}\|_{\sigma})\|\delta_{\bm{B}}\| + \|\delta_{\bm{E}}\|),
\]
where $\mathbb{C}$ depends on all the parameters $R_m, R_e, \S$ and $\|\bm{f}\|_{-1}$.
\end{theorem}


\section{Conclusion}
We analyzed a mixed finite element scheme for the stationary MHD system where both the electric and the magnetic fields were discretized on a discrete de Rham complex. Two types of boundary conditions were considered.  We rigorously established the well-posedness and proved the convergence of the finite element schemes based on weak regularity assumptions.

The electric-magnetic mixed formulation (also see \cite{hu2014stable, hu2015structure}) and the technical tools developed in this paper may also be useful for a broader class of plasma models and numerical methods, for example, compressible MHD models and discontinuous Galerkin methods (c.f. \cite{qiu2017analysis, li2005locally, warburton1999discontinuous, shadid2010towards}).

The theoretical analysis in this paper also lays a foundation for further investigation of block preconditioners for stationary MHD systems (c.f. \cite{ma2016robust, cyr2013new}).



\bibliographystyle{plain}      
\bibliography{BEconvergence}{}   

\begin{thebibliography}{10}

\bibitem{adler2018vector}
James~H Adler, Yunhui He, Xiaozhe Hu, and Scott~P MacLachlan.
\newblock Vector-potential finite-element formulations for two-dimensional
  resistive magnetohydrodynamics.
\newblock {\em Computers \& Mathematics with Applications}, 2018.

\bibitem{Arnold.D;Falk.R;Winther.R.2006a}
Douglas~N Arnold, Richard~S Falk, and Ragnar Winther.
\newblock Finite element exterior calculus, homological techniques, and
  applications.
\newblock {\em Acta numerica}, 15:1--155, 2006.

\bibitem{Arnold.D;Falk.R;Winther.R.2010a}
Douglas~N Arnold, Richard~S Falk, and Ragnar Winther.
\newblock Finite element exterior calculus: from hodge theory to numerical
  stability.
\newblock {\em Bulletin of the American mathematical society}, 47(2):281--354,
  2010.

\bibitem{badia2013unconditionally}
Santiago Badia, Ramon Codina, and Ramon Planas.
\newblock On an unconditionally convergent stabilized finite element
  approximation of resistive magnetohydrodynamics.
\newblock {\em Journal of Computational Physics}, 234:399--416, 2013.

\bibitem{Boffi.D;Brezzi.F;Fortin.M.2013a}
Daniele Boffi, Franco Brezzi, and Michel Fortin.
\newblock {\em Mixed finite element methods and applications}, volume~44.
\newblock Springer, 2013.

\bibitem{Bossavit.A.1998a}
Alain Bossavit.
\newblock {\em Computational electromagnetism: variational formulations,
  complementarity, edge elements}.
\newblock Academic Press, 1998.

\bibitem{Brackbill.J;Barnes.D.1980a}
Jeremiah~U Brackbill and Daniel~C Barnes.
\newblock The effect of nonzero {$\nabla\cdot {B}$} on the numerical solution
  of the magnetohydrodynamic equations.
\newblock {\em Journal of Computational Physics}, 35(3):426--430, 1980.

\bibitem{brenner2008mathematical}
Susanne~C Brenner and Ridgway Scott.
\newblock {\em The mathematical theory of finite element methods}, volume~15.
\newblock Springer Science \& Business Media, 2008.

\bibitem{brezzi1985two}
Franco Brezzi, Jim Douglas~Jr, and L~Donatella Marini.
\newblock {Two families of mixed finite elements for second order elliptic
  problems}.
\newblock {\em Numerische Mathematik}, 47(2):217--235, 1985.

\bibitem{cyr2013new}
Eric~C Cyr, John~N Shadid, Raymond~S Tuminaro, Roger~P Pawlowski, and Luis
  Chac{\'o}n.
\newblock {A new approximate block factorization preconditioner for
  two-dimensional incompressible (reduced) resistive MHD}.
\newblock {\em SIAM Journal on Scientific Computing}, 35(3):B701--B730, 2013.

\bibitem{Dai.W;Woodward.P.1998b}
Wenlong Dai and Paul~R Woodward.
\newblock On the divergence-free condition and conservation laws in numerical
  simulations for supersonic magnetohydrodynamical flows.
\newblock {\em The Astrophysical Journal}, 494(1):317, 1998.

\bibitem{dong2014convergence}
Xiaojing Dong, Yinnian He, and Yan Zhang.
\newblock Convergence analysis of three finite element iterative methods for
  the {2D/3D} stationary incompressible magnetohydrodynamics.
\newblock {\em Computer Methods in Applied Mechanics and Engineering},
  276:287--311, 2014.

\bibitem{gerbeau2000stabilized}
Jean-Fr\'ed\'eric Gerbeau.
\newblock {A stabilized finite element method for the incompressible
  magnetohydrodynamic equations}.
\newblock {\em Numerische Mathematik}, 87(1):83--111, 2000.

\bibitem{Girault.V;Raviart.P.1986a}
Vivette Girault and Pierre-Arnaud Raviart.
\newblock {\em Finite element methods for {Navier-Stokes} equations: theory and
  algorithms}, volume~5.
\newblock Springer Science \& Business Media, 2012.

\bibitem{greif2010mixed}
Chen Greif, Dan Li, Dominik Sch{\"o}tzau, and Xiaoxi Wei.
\newblock A mixed finite element method with exactly divergence-free velocities
  for incompressible magnetohydrodynamics.
\newblock {\em Computer Methods in Applied Mechanics and Engineering},
  199(45-48):2840--2855, 2010.

\bibitem{gunzburger1991existence}
Max~D Gunzburger, Amnon~J Meir, and Janet~S Peterson.
\newblock {On the existence, uniqueness, and finite element approximation of
  solutions of the equations of stationary, incompressible
  magnetohydrodynamics}.
\newblock {\em Mathematics of Computation}, 56(194):523--563, 1991.

\bibitem{Hiptmair.R.2002a}
Ralf Hiptmair.
\newblock Finite elements in computational electromagnetism.
\newblock {\em Acta Numerica}, 11:237--339, 2002.

\bibitem{hiptmair2018fully}
Ralf Hiptmair, Lingxiao Li, Shipeng Mao, and Weiying Zheng.
\newblock A fully divergence-free finite element method for magnetohydrodynamic
  equations.
\newblock {\em Mathematical Models and Methods in Applied Sciences},
  28(04):659--695, 2018.

\bibitem{hiptmair2018splitting}
Ralf Hiptmair and Cecilia Pagliantini.
\newblock Splitting-based structure preserving discretizations for
  magnetohydrodynamics.
\newblock {\em SMAI Journal of Computational Mathematics}, 4:225--257, 2018.

\bibitem{hu2014stable}
Kaibo Hu, Yicong Ma, and Jinchao Xu.
\newblock {Stable finite element methods preserving $\nabla\cdot\bm{B}=0$
  exactly for MHD models}.
\newblock {\em Numerische Mathematik}, 135(2):371--396, 2017.

\bibitem{hu2015structure}
Kaibo Hu and Jinchao Xu.
\newblock Structure-preserving finite element methods for stationary mhd
  models.
\newblock {\em Mathematics of Computation}, 2018.

\bibitem{li2005locally}
Fengyan Li and Chi-Wang Shu.
\newblock {Locally divergence-free discontinuous Galerkin methods for MHD
  equations}.
\newblock {\em Journal of Scientific Computing}, 22(1):413--442, 2005.

\bibitem{ma2016robust}
Yicong Ma, Kaibo Hu, Xiaozhe Hu, and Jinchao Xu.
\newblock {Robust preconditioners for incompressible MHD models}.
\newblock {\em Journal of Computational Physics}, 316:721--746, 2016.

\bibitem{Nedelec.J.1980a}
Jean-Claude N{\'e}d{\'e}lec.
\newblock {Mixed finite elements in $\mathbb{R}^{3}$}.
\newblock {\em Numerische Mathematik}, 35:315--341, 1980.

\bibitem{Nedelec.J.1986a}
Jean-Claude N{\'e}d{\'e}lec.
\newblock {A new family of mixed finite elements in $\mathbb{R}^{3}$}.
\newblock {\em Numerische Mathematik}, 50:57--81, 1986.

\bibitem{qiu2017analysis}
Weifeng Qiu and Ke~Shi.
\newblock {Analysis of a Mixed Discontinuous Galerkin method for incompressible
  magnetohydrodynamics}.
\newblock {\em arXiv preprint arXiv:1702.01473}, 2017.

\bibitem{Raviart.P;Thomas.J.1977a}
Pierre-Arnaud Raviart and Jean-Marie Thomas.
\newblock A mixed finite element method for second order elliptic problems.
\newblock {\em Lecture Notes in Mathematics}, 606:292--315, 1977.

\bibitem{Schotzau.D.2004a}
Dominik Sch{\"o}tzau.
\newblock Mixed finite element methods for stationary incompressible
  magnetohydrodynamics.
\newblock {\em Numerische Mathematik}, 96(4):771--800, 2004.

\bibitem{shadid2010towards}
John~N Shadid, Roger~P Pawlowski, Jeffrey~W Banks, Luis Chac{\'o}n, Paul~T Lin,
  and Raymond~S Tuminaro.
\newblock {Towards a scalable fully-implicit fully-coupled resistive MHD
  formulation with stabilized FE methods}.
\newblock {\em Journal of Computational Physics}, 229(20):7649--7671, 2010.

\bibitem{shadid2016scalable}
John~N Shadid, Roger~P Pawlowski, Eric~C Cyr, Raymond~S Tuminaro, Luis
  Chac\'on, and PD~Weber.
\newblock {Scalable implicit incompressible resistive MHD with stabilized FE
  and fully-coupled Newton-Krylov-AMG}.
\newblock {\em Computer Methods in Applied Mechanics and Engineering},
  304:1--25, 2016.

\bibitem{warburton1999discontinuous}
Timothy Warburton and George~Em Karniadakis.
\newblock {A discontinuous Galerkin method for the viscous MHD equations}.
\newblock {\em Journal of computational Physics}, 152(2):608--641, 1999.

\bibitem{ZhangHeYang}
Guo-Dong Zhang, Yinnian He, and Di~Yang.
\newblock Analysis of coupling iterations based on the finite element method
  for stationary magnetohydrodynamics on a general domain.
\newblock {\em Computers $\&$ Mathematics with Applications}, 68(7):770--788,
  2014.

\end{thebibliography}

\end{document}